\documentclass[12pt]{amsart}
\usepackage[utf8x]{inputenc}
\usepackage{amsmath}
\usepackage{amsfonts,mathtools}
\usepackage{amssymb}
\usepackage{amsthm,mathtools}
\usepackage{color}
\usepackage{xcolor}
\usepackage{esint}
\usepackage{enumitem}
\usepackage{geometry}

\usepackage[colorlinks=true,linkcolor=colorref,citecolor=colorcita,urlcolor=colorweb]{hyperref}
\definecolor{colorcita}{RGB}{21,86,130}
\definecolor{colorref}{RGB}{5,10,177}
\definecolor{colorweb}{RGB}{177,6,38}

\newtheorem{theorem}{Theorem}[section]
\newtheorem{lemma}[theorem]{Lemma}

\newtheorem{corollary}[theorem]{Corollary}

\theoremstyle{definition}
 
\newtheorem{remark}[theorem]{Remark}
\newtheorem{definition}[theorem]{Definition}
\numberwithin{equation}{section}
\newcommand{\N}{\mathbb{N}}
\newcommand{\C}{\mathbb{C}}
\newcommand{\R}{\mathbb{R}}

\newcommand{\E}{\mathbf{E}}
\renewcommand{\P}{\mathbf{P}}
\newcommand{\pp}{\mathbf{P}_n^{\beta_n}}

\newcommand{\lbld}{D_{BL}^-}
\newcommand{\ubld}{D_{BL}^+}
\newcommand{\bld}{D_{BL}}

\newcommand{\ww}{\mathcal{W}}
\newcommand{\sker}{\mathcal{S}}
\newcommand{\aker}{\mathcal{A}}

\newcommand{\eps}{\varepsilon}

\DeclareMathOperator{\V}{Var}

\DeclareMathOperator{\tr}{tr}
\DeclareMathOperator{\oo}{O}
\DeclareMathOperator{\so}{o}
\DeclareMathOperator{\ai}{Ai}
\DeclareMathOperator{\id}{id}
\DeclareMathOperator{\spa}{span}

\author{Yacin Ameur}
\address{Department of Mathematics \\ Lund University \\ 22100 Lund, Sweden}
\email{yacin.ameur@math.lu.se}

\author{Felipe Marceca}
\address{Department of Mathematics\\
	King's College London\\ WC2R 2LS, United Kingdom}
\email{felipe.marceca@kcl.ac.uk}

\author{Jos\'e Luis Romero}
\address{Faculty of Mathematics \\
	University of Vienna\\
	Oskar-Morgenstern-Platz 1,
	A-1090 and Acoustics Research Institute\\ Austrian Academy of Sciences\\ Dr. Ignaz Seipel-Platz 2,	AT-1010 Vienna, Austria}
\email{jose.luis.romero@univie.ac.at}

\title[The perfect freezing transition]
{Gaussian beta ensembles: the perfect freezing transition and its characterization in terms of Beurling-Landau densities}

\thanks{J. L. R. and F. M. gratefully acknowledge support from the Austrian Science Fund (FWF): 10.55776/Y1199. F. M. was also supported by the EPSRC: NIA EP/V002449/1.}

\subjclass[2010]{60K35, 82B26, 94A20, 31C20.}

\keywords{Gaussian ensemble, perfect freezing, separation, equidistribution, discrepancy, random quadrature.}

\begin{document}
\begin{abstract} The Gaussian $\beta$-ensemble is a real $n$-point configuration $\{x_j\}_1^n$ picked randomly with respect to the Boltzmann factor $e^{-\frac \beta 2 H_n}$, where $H_n=\sum_{i\ne j}\log\frac 1 {|x_i-x_j|}+n\sum_{i=1}^n\tfrac 12 x_i^2 .$ It is well known that the point process $\{x_j\}_1^n$ tends to follow the semicircle law $\sigma(x)=\tfrac {1} {2\pi}\sqrt{(4-x^2)_+}$ in certain average senses.
	
	A Fekete configuration (minimizer of $H_n$) is spread out in a much more uniform way in the interval $[-2,2]$ with respect to the regularization
	$\sigma_n(x)=\max\{\sigma(x),n^{-\frac 1 3}\}$ of the semicircle law. In particular, Fekete configurations are ``equidistributed'' with respect to $\sigma_n(x)$, in a certain technical sense of Beurling-Landau densities. 
	
	We consider the problem of characterizing sequences $\beta_n$ of inverse temperatures, which guarantee almost sure equidistribution as $n\to\infty$. We find that a necessary and sufficient condition is that $\beta_n$ grows at least logarithmically in $n$:
	$$\beta_n\gtrsim \log n.$$ We call this growth rate the \textit{perfect freezing regime}.
	Along the way,
	we give several further results on the distribution of particles when $\beta_n\gtrsim\log n$, for example on minimal spacing and discrepancies, and that with high probability a random sample solves certain sampling and interpolation problems for weighted polynomials. (In this context, Fekete sets correspond to $\beta\equiv \infty$.)
	
	The condition $\beta_n\gtrsim\log n$ was introduced in earlier works due to some of the authors in the context of two-dimensional Coulomb gas ensembles, where it is shown to be sufficient for equidistribution.
	Interestingly, although the technical implementation requires some considerable modifications, the strategy from dimension two adapts well to prove sufficiency also for one-dimensional Gaussian ensembles. On a technical level, we use estimates for weighted polynomials due to
	Levin, Lubinsky, Gustavsson and others. The other direction (necessity) involves estimates due to Ledoux and Rider on the distribution of particles which fall near or outside the boundary.
\end{abstract}

\maketitle

\section{Introduction}
\subsection{Background and preliminary discussion} In this article we study the standard Gaussian $\beta$-ensemble. This is the $n$-point process $\{x_j\}_{j=1}^n$ on $\mathbb{R}$ with
the joint PDF
\begin{align}
	\label{eqprob}
	p_n^{\beta}({x}_1, \ldots, {x}_n)= \frac{1}{Z_{{\beta},n}} \prod_{k< j} |{x}_k - {x}_j|^{\beta} e^{-n\frac{\beta}{2} \sum_{j=1}^n \frac  1 2 {x}_j^2}=
	\frac{1}{Z_{{\beta},n}} e^{-\frac{\beta}{2} H_n},
\end{align}
where $H_n$ is the Hamiltonian in Gaussian external potential $V(x)=\frac 1 2 x^2$, i.e.,
\begin{align}\label{ham1}H_n({x}_1, \ldots, {x}_n)= \sum_{k\ne j} \log\frac{1}{|{x}_k - {x}_j|} -  n \sum_{j=1}^n V(x_j),\qquad V(x)=\tfrac 1 2 x^2.\end{align}
We write $\P_n^\beta$ for the probability measure on $\R^n$ with density $p_n^\beta$ with respect to Lebesgue measure on $\R^n$.

It is useful to think of an $n$-point configuration $\{x_j\}_1^n$ as a system of point charges on $\R$. The first term on the right in \eqref{ham1} represents the interaction energy between pairs of particles, and the second term represents the energy coming from
interaction with the external field $nV$; the parameter $\beta>0$ has the meaning of an inverse temperature.
For $\beta=1,2,4$, we obtain respectively the classical Gaussian orthogonal, unitary and symplectic ensembles (GOE, GUE and GSE).

Energy-minimizers $\{{x}_j\}_1^n$, which render $H_n$ minimal, are known as (weighted) Fekete sets.
In a sense, such configurations correspond to the formal zero temperature limit, i.e., to $\beta=\infty$.
For the Gaussian potential, Fekete configurations are unique and coincide with the zeros of a suitably scaled $n$-th Hermite polynomial, see \cite{Ism,Meh}.

A natural counterpart of the discrete energy $H_n$ for continuous charge distributions (measures) $\mu$ on $\R$ is given by
\[I_V[\mu]=\iint_{\R^2} \log\frac{1}{|x-y|}\, d\mu(x) d\mu(y)+\int_\R V(x)\, d\mu(x),\]

A well-known theorem of potential theory (e.g.~\cite[Proposition 6.156]{De99} or \cite[Theorem~IV.5.1]{SaTo}) asserts that the semicircular probability density
\begin{align}\label{eq_sc}
	\sigma(x)= \tfrac{1}{2\pi} \sqrt{(4-x^2)_+},\qquad ((r)_+=\max\{r,0\}),
\end{align}
minimizes the energy $I_V[\mu]$ over all compactly supported Borel probability measures $\mu$ on $\mathbb{R}$.
In the context of log-gases, the support $[-2,2]$ of $\sigma$ is known as the \emph{droplet}.

For general reasons \cite{De99,SaTo} Fekete sets $\{x_j\}_1^n$ are contained in the droplet, and the probability
measures $\frac 1 n\sum_{j=1}^n \delta_{x_j}$ converge weakly to $\sigma$ as $n\to\infty$. In addition,
due to their characterization as zeros of Hermite polynomials, the position of Fekete points is known with $O(n^{-1})$ precision, see \cite{KrML}, and \cite{DKMLV,MNT,Gu}.

Now fix a finite $\beta>0$. Given a random sample $\{x_j\}_1^n$ from $\P_n^\beta$ we
denote the number of particles which fall in an interval $I$ by
\begin{align*}
	\#[I]=\#I\cap\{x_j\}_{j=1}^n.
\end{align*}
Also, we write $B(x,r)=(x-r,x+r)$ and denote the \emph{one-point intensity} by
\[R_n^\beta(x)=\lim_{\varepsilon\to 0} \frac{\E_n^\beta\#
	[B(x,\varepsilon)]}{2\varepsilon},\]
where $\E_n^\beta$ is expectation with respect to $\P_n^\beta$.

We then have the convergence (``Wigner semicircle law'')
$$\tfrac{1}{n}R_n^\beta\, dx\to d\sigma,\qquad  (n\to\infty)$$ in the weak sense of measures, see \cite[Theorem~2.1]{Jo}. In addition, the fluctuations about the equilibrium
\[\sum_{j=1}^n f(x_j)-n\int_{\mathbb{R}} f\,d\sigma,\]
have been well studied and are known to converge in distribution as $n\to\infty$
to a normal distribution with known expectation and variance if $f$ is smooth, \cite{Jo,BeLeSe,Sh,LaLeWe,F22}. These fluctuation results suggest that a random sample from $\P_n^\beta$ should tend to be, in some sense, more ``disorganized'' than Fekete sets. We shall now make this intuition more precise.

\subsection{Equidistribution and the perfect freezing regime} We shall introduce a suitable notion of equidistribution of a family of configurations,
which is strong enough to distinguish Fekete sets from Gaussian $\beta$-ensembles. In fact, we shall find that the threshold level is precisely the perfect freezing regime $\beta_n\gtrsim \log n$.

For this purpose, we start by defining the \textit{regularized semicircle density}
\begin{align}\label{eq_psi}
	\sigma_n(x)= \max\{\sigma(x),n^{-\frac 1 3}\}.
\end{align}
(The values of $\sigma_n(x)$ for $x$ outside of a vicinity of the droplet $[-2,2]$ will be irrelevant and $\sigma_n(x)$ could equally be redefined as zero there.)

The quantity \begin{equation}\label{munn}\mu_n(x)=\frac 1{n\sigma_n(x)}\end{equation} gives a suitable notion of microscopic scale in a vicinity of the droplet. We may think about $\mu_n(x)$ as an approximate
expected spacing between neighbouring particles near the point $x$.

Given an arbitrary real configuration $\{x_j\}_1^n$ and a point $x$ in the droplet, we blow up about $x$ on the $\mu_n$-scale and form the
rescaled configuration $\{\tilde{x}_j\}_1^n$ by
\begin{equation}\label{voxd}\tilde{x}_j=n\sigma_n(x)\cdot (x_j-x).\end{equation}
(This scaling is perhaps particularly natural from a point of view of the GUE, since then $n\sigma_n(x)$ is equivalent, up to constants, to the density $K_n(x,x)$. See \eqref{voxday} below.)

If $\{x_j\}_1^n$ are Fekete points, then $\{\tilde{x}_j\}_1^n$ in \eqref{voxd} are very well spread out on the varying scale $\mu_n(x)$. See Figure \ref{fig1}. (We thank S.~Byun for help with the figures.)

\begin{figure}[ht]
	\begin{center}
		\includegraphics[scale=0.11]{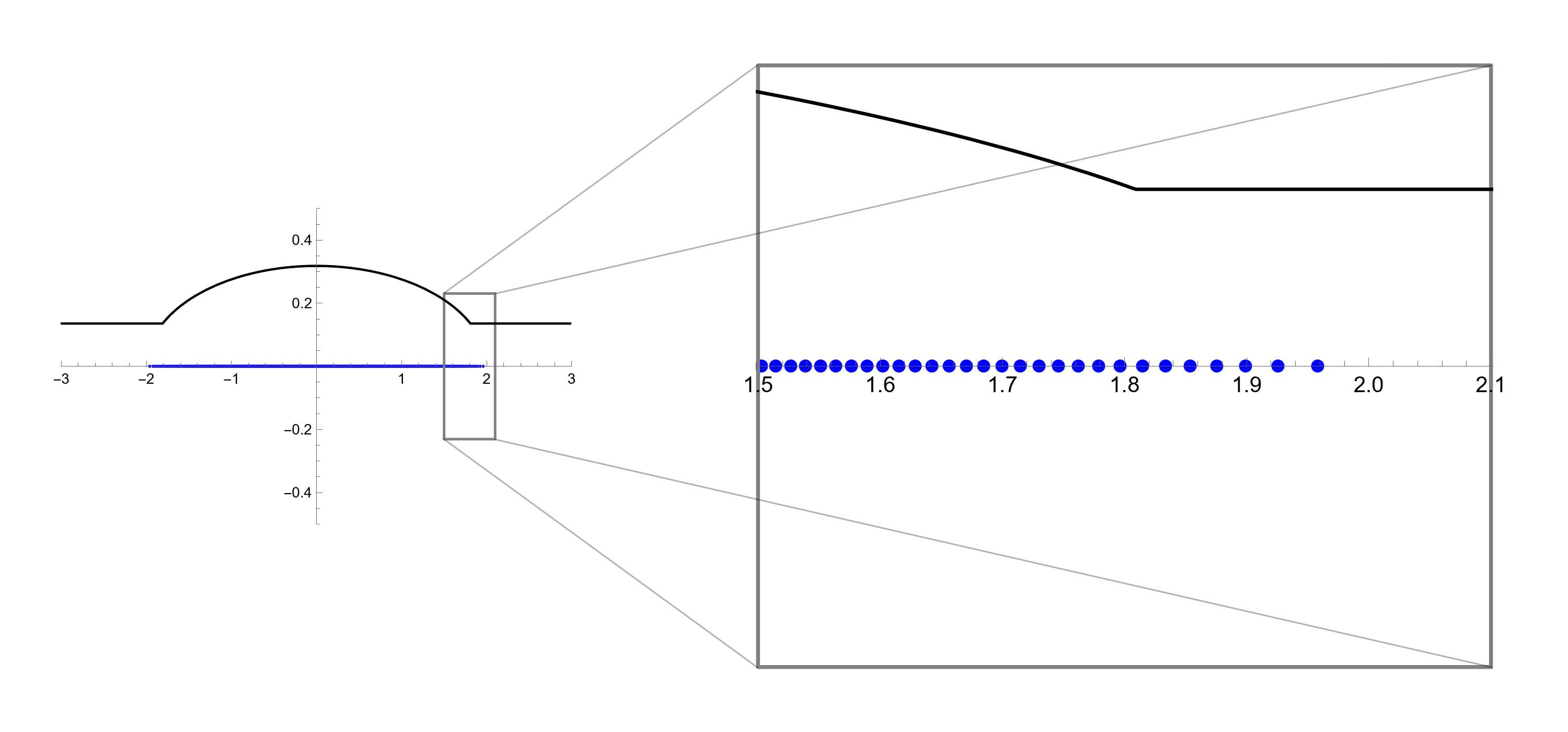}
	\end{center}
	\caption{A numerically generated Fekete set with $n=400$ points. The graph depicts the function
		$\sigma_n(x)$ for $n=400$.}
	\label{fig1}
\end{figure}
\begin{figure}[ht]
	\begin{center}
		\includegraphics[scale=0.11]{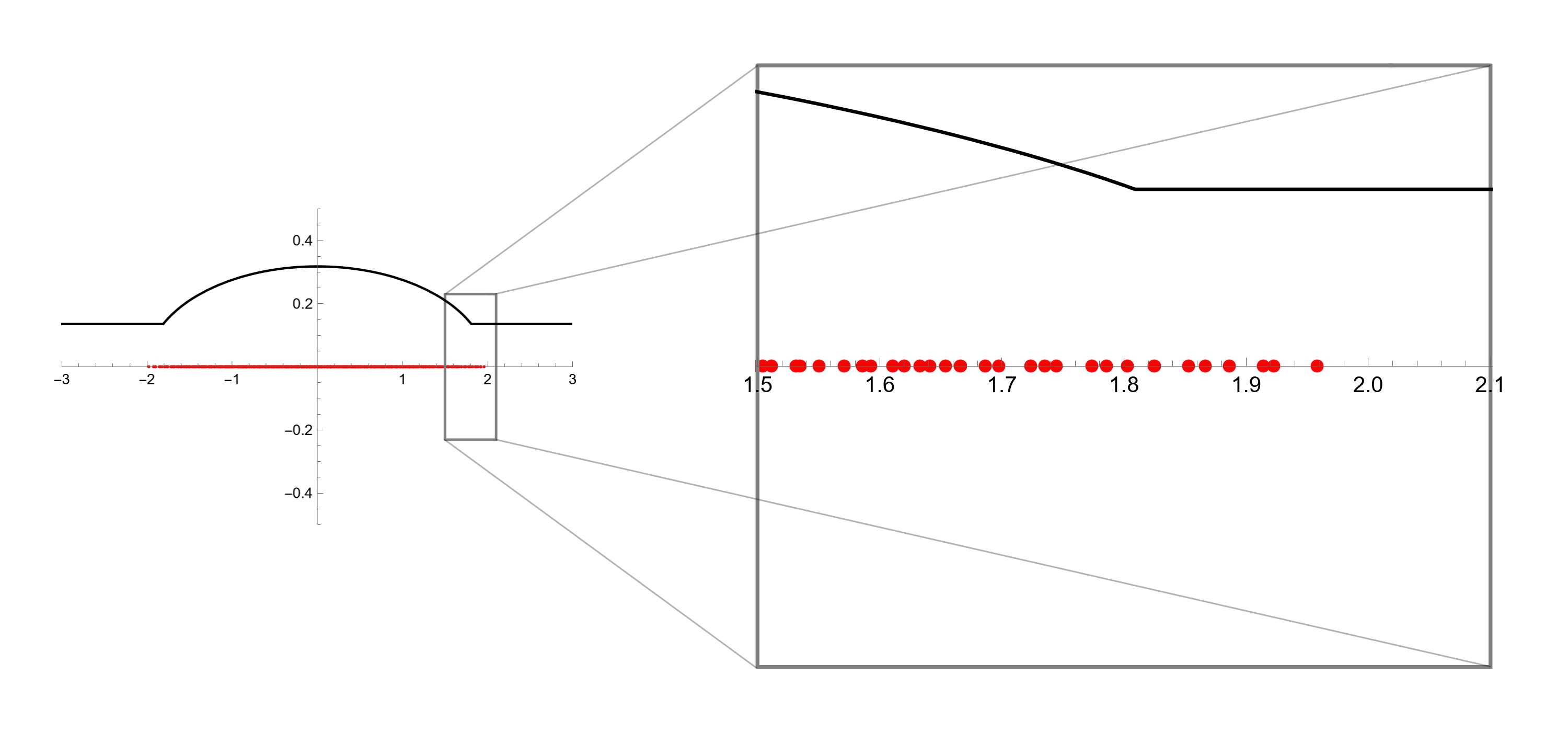}
	\end{center}
	\caption{A sample from the GUE, with $n=400$ points.}
	\label{fig2}
\end{figure}

Rescaled samples $\{\tilde{x}_j\}_1^n$ picked with respect to
$\P_n^\beta$ have also been well studied, but they behave more wildly than Fekete sets, see Remark \ref{sens0} as well as Figure \ref{fig2}.
On the other hand,
on \textit{mesoscopic scales} which are $n^\varepsilon$ times larger than $\mu_n$, things look much nicer, see e.g.
\cite{BaSu,BeLo,BoErYa,Gu,SoWo,Wo}.  We now briefly describe some aspects of this development, but first we fix some notation.

We shall work with \textit{families} of configurations
$$x=\{x^{(n)}\}_{n=1}^\infty$$
where each $x^{(n)}=\{x^{(n)}_j\}_{j=1}^n$ is an $n$-point configuration. A natural probability distribution on the set of families is given by the product measure
$$\mathbf{P}=\prod_{n=1}^\infty \mathbf{P}_n^{\beta_n},$$ where $\beta_n$ are positive numbers which may depend on $n$. To ease up the notation we usually write $x_j$ rather than $x_j^{(n)}$.

In the following, the term \textit{almost surely}
will always be understood with respect to the measure $\mathbf{P}$, except when the opposite is made explicit.

We now introduce our notion of Beurling-Landau densities, which is adapted to
the multifractal nature of families picked randomly with respect to the probability law $\mathbf{P}$.

\begin{definition} Let $(x^{(n)})_{n=1}^\infty$ be a family of $n$-point configurations and write $x^{(n)}=\{x^{(n)}_j\}_{j=1}^n$. Given an interval $I$ we write
	$$\#_n[I]=\#(I\cap\{x^{(n)}_j\}_{j=1}^n).$$
	Also let $x$ be a point with $-2\le x\le 2$ and $\varepsilon\ge0$.
	We define the following lower and upper Beurling-Landau densities of the family at $x$.
	\begin{enumerate}
		\item If $x$ is in the bulk, we define
		\begin{align}\label{blbulk}
			\begin{aligned}
				\bullet && \ubld(x,\varepsilon)&:=\limsup_{L\to \infty}\limsup_{n\to\infty}
				\frac{\#_n\big[B(x,L n^\varepsilon/n)\big]}{2L n^\varepsilon}
				\\ \bullet && \lbld(x,\varepsilon)&:=\liminf_{L\to \infty}\liminf_{n\to\infty}
				\frac{\#_n\big[B(x,L n^\varepsilon/n)\big]}{2L n^\varepsilon}
			\end{aligned}\qquad (-2<x<2).
		\end{align}
		\item If $x$ is at the edge we define
		\begin{align}\label{bledge}
			\begin{aligned}
				\bullet && \ubld(x,\varepsilon)&:=\limsup_{L\to \infty}\limsup_{n\to\infty}
				\frac{\#_n\big[B(x,(L n^\varepsilon/n)^{\frac 23})\big]}{L n^\varepsilon}
				\\ \bullet && \lbld(x,\varepsilon)&:=\liminf_{L\to \infty}\liminf_{n\to\infty}
				\frac{\#_n\big[B(x,(L n^\varepsilon/n)^{\frac 23})\big]}{L n^\varepsilon}
			\end{aligned}\qquad (x=\pm 2).
		\end{align}
	\end{enumerate}
	If both densities coincide we call the common value the Beurling-Landau density at scale $\varepsilon$ and write
	\[\bld(x,\varepsilon):=\ubld(x,\varepsilon)=\lbld(x,\varepsilon).\]
	For the \textit{microscopic scale} $\varepsilon=0$ we just write $\bld(x),\ubld(x),\lbld(x)$ for the respective Beurling-Landau densities:
	$$\ubld(x)=\ubld(x,0),\quad \lbld(x)=\lbld(x,0),\quad \bld(x)=\bld(x,0).$$
	If $\varepsilon>0$ we speak of a \textit{mesoscopic scale}. (In this case the parameter $L$ does not play a crucial role, but facilitates the exposition.)
\end{definition}

In the following the sample-size $n$ is always understood. To simplify the notation we will therefore drop the sub-index in the sequel, writing $\#[I]$ in place of $\#_n[I]$.

\begin{remark} To define Beurling-Landau densities, we zoom on $x$ by blowing up by the factor $n$ if $x$ is in the bulk and $n^{\frac 2 3}$ if $x\in \{\pm 2\}$.
	For the rescaled process, we count the number of points in an appropriate large interval (either $[-Ln^\varepsilon,Ln^\varepsilon]$ or $[-(Ln^\varepsilon)^{\frac 2 3},(Ln^\varepsilon)^{\frac 2 3}]$) and normalize suitably, by a factor proportional to $Ln^\varepsilon$.
	Finally we send $n\to\infty$, then $L\to\infty$. We shall find that $\bld$, when it exists, gives a suitable, \textit{global} measure of equidistribution for samples from Gaussian $\beta$-ensembles.
\end{remark}

We now survey some relevant results on mesoscopic and microscopic densities in the case of a fixed finite $\beta$ (i.e. independent of $n$).

For $x$ in the bulk $(-2,2)$ and $\eps>0$ arbitrarily small, the following mesoscopic result follows from \cite[Theorem~3.1]{BoErYa}:
\begin{align}\label{emes}
	\bld(x,\varepsilon)= \sigma(x) \qquad  \text{almost surely.}
\end{align}
A similar statement found in \cite{Wo} applies at the boundary:
\begin{align}\label{nemes}
	\bld(x,\varepsilon)=\frac 2 {3\pi}\qquad  \text{almost surely.}
\end{align}
However, \eqref{emes} and \eqref{nemes} are \textit{false} at the microscopic scale $\varepsilon=0$. For $x$ in the bulk it follows from \cite[Theorem~1]{HoVa} (details are given in Theorem \ref{necbl} below) that for every fixed $\beta>0$ and every $L>0$,
\[\limsup_{n\to \infty} \#[B(x,L/n)]= \infty \qquad  \text{almost surely,}\]
contradicting \eqref{emes} with $\eps=0$.
At the edge we can use results due to Ledoux and Rider (see Theorem \ref{necbl} below) to deduce that
$$\liminf_{n\to \infty}\#[B(\pm 2,(L/n)^{\frac 2 3})]=0\qquad  \text{almost surely}$$
for each (large) $L$. This contradicts \eqref{nemes} at $\eps=0$.

\begin{remark} \label{sens0} While the above facts show that it is impossible to make sense of microscopic Beurling-Landau densities for bona fide Gaussian $\beta$-ensembles, it is well-known that one can
	define limiting point-processes on the microscopic scale, namely $\text{Sine}_\beta$  in the bulk and  $\text{Airy}_\beta$  at the edge, see \cite{VaVi,RaRiVi}. These processes are in fact used in several instances below. We note that relevant discrepancy estimates for $\text{Sine}_\beta$ and  $\text{Airy}_\beta$ can be found in \cite{HoPa,Zh}.
\end{remark}

In what follows we will frequently encounter
the following condition: there exists a constant $c_0>0$ such that (for all large $n$)
\begin{align}\label{eq_lowtemp}
	\beta={\beta_n} \geq c_0 \log n.
\end{align}
We refer to it as the \textit{perfect freezing regime}.
Under this condition we shall find that as $n\to\infty$, the system $\{x_j\}_1^n$ exhibits a freezing behavior and inherits some properties of Fekete sets. Moreover, \eqref{eq_lowtemp} is the condition we need in order to
take the equidistribution results \eqref{emes} and \eqref{nemes} to the microscopic scale $\eps=0$.

The following remark gives some background on our choice of terminology.

\begin{remark} The condition \eqref{eq_lowtemp} appears in a two-dimensional setting in the paper \cite{Am}. More recently, in \cite{AmRom}, it is proved that \eqref{eq_lowtemp}
	implies global equidistribution for a large class of two-dimensional Coulomb-gas ensembles, and the condition is there conjectured to be ``sharp''. It is relevant to note that, in dimension two, a different kind of phase transition is
	believed to occur when $\beta$ increases beyond a value of approximately $\beta_f\approx 140$, see \cite{CSA} and comments in \cite[Section 7]{AmRom}. For $\beta>2$ certain oscillations of the 1-point density appear near the boundary.
	It is believed that, as $\beta$ increases beyond $\beta_f$, the oscillations start propagating towards the bulk at an increasing rate, where they eventually freeze up in the form of Fekete-configurations at $\beta=\infty$. Hence the assumed transition-value $\beta_f$ might be called an ``inverse freezing temperature''. On the other hand, equidistribution in the sense of Beurling-Landau densities gives a different kind of phase transition, which for lack of better was termed ``perfect freezing'' in \cite{AmRom}, to distinguish it from the supposed transition $\beta_f$. It is proved in \cite{AmRom} that the perfect freezing transition in dimension 2 (call it $\beta_p$) satisfies $\beta_p\lesssim \log n$, and it is conjectured that $\beta_p\asymp\log n$.
	In any event, we shall verify below that the perfect freezing regime \eqref{eq_lowtemp} is necessary and sufficient in order to achieve equidistribution in the context of one-dimensional Gaussian ensembles.
\end{remark}

We are now ready to state our main results on equidistribution.

\begin{theorem}[Sufficiency for equidistribution] \label{sufbl}
	If for some constant $c_0>0$
	$$\beta_n\ge c_0 \log n,$$ then almost surely the following statements hold simultaneously:
	\begin{enumerate}[label=(\roman*),ref=(\roman*), itemsep=1em]
		\item \label{bulk} (Bulk) For every $x\in (-2,2)$, the bulk Beurling-Landau density \eqref{blbulk} with $\varepsilon=0$
		exists and equals to the semi-circle density:
		\begin{align*}
			\bld(x)=\sigma(x),\qquad -2<x<2;
		\end{align*}
		Moreover, the limits are uniform in the following sense:
		\begin{align*}
			\lim_{L\to \infty}\sup_{\delta>0}\limsup_{n\to\infty}
			\sup_{x \in [-2+\delta,2-\delta]}\Big|\frac{\#\big[B(x,L/n)\big]}{2L }-\sigma(x)\Big|=0.
		\end{align*}
		\item \label{boun} (Boundary)  For $x=\pm 2$ the edge Beurling-Landau density \eqref{bledge} with $\varepsilon=0$
		exists and equals to
		\begin{align*}
			\bld(x)=\frac{2}{3\pi};
		\end{align*}
		\item \label{exte} (Exterior)
		There exists a constant $M\ge 1$ depending on $c_0$ such that,
		\begin{align*}
			\lim_{n\to\infty}
			\#\big[\,\R\setminus ((2+Mn^{-\frac 23})[-1,1])\,\big]=0.
		\end{align*}
	\end{enumerate}
\end{theorem}

We shall say that a sample is \textit{equidistributed} if it simultaneously obeys all three properties \ref{bulk},\ref{boun},\ref{exte}. Thus Theorem \ref{sufbl} says
that the condition $\beta_n\gtrsim\log n$ implies equidistribution of almost every random family. We remark that the third statement \ref{exte} immediately follows from \cite[Theorem~1]{LeRi} (see Theorem~\ref{lemmaprob}), whereas the first two statements \ref{bulk} and \ref{boun} are a consequence of more precise discrepancy estimates which we present in Theorem~\ref{propdisc}.

We have also a converse to Theorem \ref{sufbl}.

\begin{theorem}[Necessity for equidistribution] \label{necbl} The following statements hold.
	\begin{enumerate}[label=(\Roman*),ref=(\Roman*), itemsep=1em]
		\item \label{ZW} (Bulk) If the inverse temperature $\beta_n=\beta$ is fixed, then for every $x$ with $-2<x<2$ almost surely for every $L>0$,
		\[\limsup_{n\to \infty} \#\big[B(x,L/n)\big]=\infty
		\qquad \text{and} \qquad \liminf_{n\to \infty} \#\big[B(x,L/n)\big]=0.\]
		In particular,
		\[\ubld(x)=\infty
		\qquad \text{and} \qquad \lbld(x)=0
		\qquad \text{almost surely.}\]
		\item \label{ZE} (Boundary) If $\beta_n$ is not in the perfect freezing regime, that is if
		\begin{equation*}\liminf_{n\to\infty}\frac {\beta_n}{\log n}=0,\end{equation*}
		then for $x=\pm 2$ almost surely for every $L>0$,
		\[\liminf_{n\to \infty} \#\big[B(x,(L/n)^{\frac 23})\big]=0.\]
		In particular,
		\[\lbld(x)=0 \qquad \text{almost surely.}\]
		\item \label{ZZ} (Exterior) If $\beta_n$ is not in the perfect freezing regime, then almost surely for every $L>0$,
		\[\limsup_{n\to\infty}
		\#\big[\,\R\setminus ((2+Ln^{-\frac 23})[-1,1])\,\big]\ge 1 .\]
	\end{enumerate}
	In particular, if almost every sample is equidistributed, then necessarily $\beta_n\gtrsim \log n$.
\end{theorem}

While the contexts differ, 
the above results might be regarded as a manifestation of the principle from \cite{CSA} that \textit{freezing is most naturally studied by looking closely at the edge}.

We mention that the proof of \ref{ZW} follows from overcrowding and gap results for the $\text{Sine}_\beta$ kernel \cite{HoVa,VaVi}, while \ref{ZE} and \ref{ZZ} are a consequence of lower tail bounds for $\max_j x_j$ from Ledoux and Rider \cite{LeRi}.

In the bulk case \ref{ZW} we do not know the optimal growth rate of $\beta_n$ implying that $D_{BL}(x)=\sigma(x)$ almost surely for each $x$ in the bulk. (We know that $\beta_n\gtrsim \log n$ is sufficient for $D_{BL}(x)=\sigma(x)$ in the bulk, and $\lim_n \beta_n\to \infty$ is necessary.)

\subsection{Key auxiliary results and definitions} Before continuing with further main results, it is convenient to display some key notions and facts, and to lay out some of our main strategy.

We start by stating \cite[Theorem~1]{LeRi} due to Ledoux and Rider.

\begin{theorem}[Ledoux-Rider's localization theorem, \cite{LeRi}]\label{lemmaprob}
	Let $\{x_j\}_1^n$ be a random sample from \eqref{eqprob}. There exists a universal constant $C>0$ such that for every $n\ge 1$, $0<\varepsilon\le 1$ and $\beta\geq 1$,
	\[\mathbf{P}_n^\beta\left(\left\{\max_{1\le j\le n}|{x}_j|>2+\varepsilon\right\}\right)\le C e^{-n \beta \varepsilon^{\frac 32}/C}.\]
	
	In particular, in the perfect freezing regime
	$\beta_n\ge c_0\log n$, for every $\tau>0$ there is a constant $M\ge 1$, depending on $\tau$ and $c_0$, such that
	\begin{align}\label{erl}
		\pp\left(\left\{\max_{1\le j\le n}|{x}_j|>2+Mn^{-\frac 23}\right\}\right)\le n^{-\tau}.
	\end{align}
\end{theorem}

\begin{remark}
	In the perfect freezing regime, Theorem \ref{lemmaprob} allows us to focus our attention on a microscopic neighbourhood of the droplet. A two-dimensional counterpart appears in \cite{Am} (see also \cite{Th,AmTr}), where the result was termed ``localization theorem''. This latter result was later used systematically to study equidistribution in the paper \cite{AmRom}. Since Theorem~\ref{lemmaprob} will play a similar role, we term Theorem \ref{lemmaprob} a localization theorem as well.
\end{remark}

We next introduce some key objects, important for our proof of Theorem \ref{sufbl}.

We shall work frequently with the space of Gaussian weighted polynomials (of a \textit{real} variable $x$)
\begin{align}\label{eq_wp}
	\ww_n = \{f(x)=q(x)\cdot e^{-\frac 1 4nx^2}\,; \, q \in \mathbb{C}[x],\, \deg q \leq n-1\}.
\end{align}
We regard $\ww_n$ as a subspace of $L^2(\mathbb{R})$.

Given a configuration $\{x_j\}_{j=1}^n$ and an integer $j$ between $1$ and $n$ we define the associated (random) \textit{weighted Lagrange polynomials} $\ell_j\in \ww_n$ by
\begin{align}\label{eq_lp}
	\ell_j(x)=\Big(\prod_{k\neq j} \frac{x-x_k}{x_j-x_k} \Big) e^{-\frac 1 4 n(x^2-x_j^2)},\qquad x\in\R.
\end{align}
In what follows, a key fact is that in the perfect freezing regime, the Lagrange polynomials are almost surely eventually uniformly bounded.

\begin{theorem}\label{tub}
	In the perfect freezing regime \eqref{eq_lowtemp}, there is a constant $C\ge 1$ depending on $c_0$ such that
	\[\limsup_{n\to \infty}\max_{1\le j\le n} \|\ell_j\|_\infty \le C \qquad  \text{almost surely.}\]
\end{theorem}

We shall find that the uniform bound of Lagrange functions in Theorem \ref{tub} gives us what we need to apply a suitable variant of the Beurling-Landau circle of ideas, which is crucial for our proof of sufficiency for equidistribution (Theorem \ref{sufbl}) as well as of several more refined results (see Section \ref{sec_a} below).

Another key ingredient in our analysis is the reproducing kernel $K_n(x,y)$ for the space $\ww_n$. As is well-known, this is the canonical correlation kernel for the ($\beta=2$) GUE $n$-point process. In particular, the $1$-point intensity is just
$$R_n^{(\beta=2)}(x)=K_n(x,x).$$

For random Hermitian matrices (i.e. $\beta=2$) powerful Riemann-Hilbert techniques are available, leading to very strong asymptotic results, see \cite{De99,Kuijlaars} for nice introductions.

While it is not clear a priori that a good asymptotic knowledge of the $\beta=2$ kernel should be useful to study the distribution in the freezing regime $\beta_n\gtrsim \log n$, we shall see below that this is indeed the case.

\subsection{Separation, sampling and interpolation, and discrepancy estimates}
\label{sec_a}
In this subsection, we present a number of refined results on the nature of the distribution of
particles in the perfect freezing regime $\beta_n\ge c_0\log n$. In addition some of our results, for example on separation, carry over to give non-trivial information also for fixed $\beta$.

\subsubsection*{Separation}
Consider a random sample $\{x_j\}_1^n$ from \eqref{eqprob}.
The closest spacing between different points, or the \emph{separation}, is an important and well studied statistic. Let us quickly recast some known results.

For the $n$-Fekete set, the separation can be precisely estimated and the scaling order is exactly the microscopic scale $\mu_n(x)$ in \eqref{munn} ranging from $n^{-1}$ in the bulk to $n^{-\frac 23}$ at the boundary (see \cite[Corollary~13.4]{LeLu2}). For the GUE ($\beta=2$), the distance between two closest neighbors scales as $n^{-\frac 43}$ \cite{Vi}. This means that the GUE shows a strong repulsion in comparison to \emph{iid} $\sigma$-distributed random variables, where separation scales as $n^{-2}$. For general values of $\beta$, the exact separation order is conjectured to be $n^{-\frac{\beta+2}{\beta+1}}$ in \cite{FeTiWe}. In fact, this conjectured separation is rigorous
for the circular $\beta$-ensemble (C$\beta$E) when $\beta$ is an integer. The proof in \cite{FeWe} depends on Selberg's integral formula. (Cf.~also \cite[Exercise~14.6.5]{Fo}.)

In view of the cited conjecture, almost sure separation at the critical scale $n^{-1}$ is not expected for any fixed $\beta>0$. However, by naively extrapolating the conjectured separation $\approx n^{-\frac{\beta+2}{\beta+1}}$ to $n$-dependent betas, it could be guessed that the Fekete-like separation order $n^{-1}$ should be achievable in the perfect freezing regime. Indeed, we will verify this expected behavior.

After these preliminaries we define the normalized \textit{spacing} of a sample $\{x_j\}_1^n$ by
\begin{align}\label{eq_sep}
	\mathrm{s}_n := \min_{j\ne k} \left\{ n\sigma_n({x}_j)\cdot|{x}_j - {x}_k|\right\},
\end{align}
where $\sigma_n(x)$ is the regularized semicircle density given in \eqref{eq_psi}. From \cite[Theorem~1.1]{LeLu3} (see also Lemma~\ref{lspinoff}), it follows that
\begin{equation}\label{voxday}R_n^{(\beta=2)}(x)=K_n(x,x)\simeq n\sigma_n(x),\end{equation}
near the droplet. In other words, the spacings are normalized according to the varying background provided by the 1-point density corresponding to $\beta=2$.

\begin{theorem}[Separation]
	\label{sepa}
	There are universal constants $c,C>0$ such that for every $\beta\ge 1$, $\tau>0$ and all integers $n\ge 1$,
	\begin{align}\label{eq_sepaa}
		\mathbf{P}_n^\beta\left(\left\{\mathrm{s}_n < c n^{-\frac {2+\tau}\beta}\right\} \right)<C  (n^{-\tau}+ e^{- n \beta /C}).
	\end{align}
	In particular, for $\beta_n\ge c_0 \log n$ there is a positive constant $s_0=s_0(c_0)$ such that the sample is uniformly $s_0$-separated:
	\begin{align}
		\label{eq_sepa}
		\liminf_{n\to\infty} \mathrm{s}_n \ge s_0,\qquad \text{almost surely}.
	\end{align}
\end{theorem}

We stress that the normalization by $n\sigma_n$ in \eqref{eq_sep} allows us to study the microscopic scale \emph{uniformly along the droplet}, without having to distinguish between bulk and Airy regimes, and delivers \emph{global conclusions}. In contrast, results in the literature \cite{Fo,Vi} are usually local or not normalized (and thus provide relevant information for the bulk, but less so in the Airy regime).

Theorem \ref{sepa} plays a central role for our proof of equidistribution in the perfect freezing regime. In addition, the result has interesting implications at fixed finite $\beta$, as the following corollary shows.

\begin{corollary} Let us denote by $$\tilde{\mathrm{s}}_n=\min_{j\ne k}\{|x_j-x_k|\}$$ the non-normalized spacing of $\{x_j\}_1^n$ at fixed inverse temperature $\beta$.
	According to \eqref{eq_sepaa}, we have almost surely a lower bound of the form
	\begin{equation}\label{spo}\tilde{\mathrm{s}}_n\gtrsim n^{-\frac{\beta+2+\tau}{\beta}}.\end{equation}
\end{corollary}

Note that for large $\beta$ and small $\tau>0$, the exponent $\frac {\beta+2+\tau}\beta$ in \eqref{spo} is quite close to the conjectured optimal exponent $\frac{\beta+2}{\beta+1}$ of \cite{FeTiWe}.

\subsubsection*{Sampling and interpolation}
Next we discuss sampling and interpolation inequalities in $\ww_n$. It is convenient to first recall a few facts pertaining to Fekete sets.

Since Fekete points  $\{x^*_j\}_1^n$ coincide with the zeros of the $n$-th Hermite polynomial, it follows that they are the nodes of the following \emph{Gaussian quadrature formula} in $\ww_{2n}$:
\begin{align*}
	\int_{\mathbb{R}} f(x) dx=\sum_{j=1}^n \frac{f(x_j^*)}{K_n(x^*_j, x^*_j)}, \qquad f \in \ww_{2n}.
\end{align*}
In turn, this formula implies that Fekete points satisfy the following Gaussian isometry:
\begin{align}\label{eq_iso}
	\|f\|_2^2=\sum_{j=1}^n \frac{|f(x_j^*)|^2}{K_n(x_j^*,x_j^*)}, \qquad f \in \mathcal{W}_n.
\end{align}
Our next result provides the following sampling and interpolation relations, which show that random samples $\{x_j\}_1^n$ in the perfect freezing regime can likewise be used
to discretize the $L^2$-norm of weighted polynomials, in a probabilistic sense.

\begin{theorem}[Sampling and interpolation in the perfect freezing regime]\label{theosam}
	Draw $\{x_j\}_1^n$ randomly from \eqref{eqprob} under the perfect freezing condition \eqref{eq_lowtemp} and fix $\tau>0$.  There are positive
	constants $C$ (depending on $c_0$) and $A,M,s_0$  (depending on $c_0$ and $\tau$) such that for every $n \in \N$ with probability $\mathbf{P}_n^{\beta_n}$ $\ge 1- C(n^{-\tau}+ e^{- \beta_n n /C})$, the following properties hold simultaneously:	
	
	\begin{enumerate}[label=(\roman*),ref=(\roman*)]
		\item \label{eqwid} (Width):
		\begin{align*}
			\{x_j\}_{j=1}^n\subseteq (2+M n^{-2/3})[-1,1].
		\end{align*}
		\item \label{eqsepar} (Separation):
		\begin{align*}
			\mathrm{s}_n\ge s_0.
		\end{align*}
		\item \label{itint} (Interpolation): For each $1<\rho<2$ and each sequence of values $(a_j)_{j=1}^n\subseteq \C$
		there exists $f \in \ww_{\lceil n\rho \rceil}$ such that
		\[f(x_j) = a_j, \quad j = 1, \ldots , n,\]
		and
		\begin{align}
			\label{eqinter}
			\int_\R |f|^2 \le \frac{A}{ (\rho - 1)^2} \sum_{j=1}^n \frac{|a_j|^2}{K_n(x_j,x_j)}.
		\end{align}
		\item  \label{itsam} (Sampling): For each $0 < \rho< 1$, the following inequality holds
		\begin{align}
			\label{eqsam}
			\int_\R |f|^2 \le \frac{A}{ (1-\rho)^2} \sum_{j=1}^n \frac{|f(x_j)|^2}{K_n(x_j,x_j)}, \quad f \in \ww_{\lfloor n\rho \rfloor}.
		\end{align}
	\end{enumerate}
\end{theorem}

We remark that \ref{eqsepar} follows from Theorem~\ref{sepa}, while \ref{eqwid} follows directly from \cite[Theorem~1]{LeRi} (see Theorem~\ref{lemmaprob} above), and is restated here for completeness.

The separation condition \ref{eqsepar} readily delivers the following Bessel bound complementing \eqref{eqsam}:
\begin{align}\label{eqsam2}
	\sum_{j=1}^n \frac{|f(x_j)|^2}{K_n(x_j,x_j)} \leq B \int_\R |f|^2, \quad f \in \ww_{n},
\end{align}
for a constant $B=B(c_0)$ --- cf. Lemma~\ref{lemmabessel} below. In the approximation theory literature, the conjunction of \eqref{eqsam} and \eqref{eqsam2} is often called a \emph{Marcinkiewicz-Zygmund property}. A special example is given by the
Gaussian isometry for Fekete points \eqref{eq_iso}. To compare, the sampling relation \ref{itsam} for the Gaussian $\beta_n$-ensemble  holds only on the smaller subspace $\ww_{\lfloor n\rho \rfloor}$ with $\rho<1$, and thus, in contrast to the Fekete case, the values
\begin{align}\label{eq_a}
	f(x_1), \ldots, f(x_n)
\end{align}
are redundant. Such redundancy is however small, as the interpolation relation \ref{itint} shows that the values \eqref{eq_a} can be freely prescribed for $f \in \ww_{\lceil n\rho \rceil}$ and $\rho>1$.

It is worth pointing out that the sampling inequality in Theorem \ref{theosam} bounds the norm of a weighted polynomial on the whole line by corresponding values taken near the droplet. This is in line with the restriction inequalities for weighted
real-variable polynomials such as \cite[Theorem~4.2~(a)]{LeLu2} (Theorem~\ref{theores}). In contrast, the sampling theorem in \cite{AmRom} provides control of the norm of a complex-variable weighted polynomial only near the droplet. (This points to yet another inherent difference between the one- and two-dimensional settings.)

\subsubsection*{Discrepancy}
We now turn to the discrepancy between the number of points and the equilibrium measure at microscopic scales. We remind that the semicircle density $\sigma(x)$ is given in \eqref{eq_sc}.

\begin{theorem}[Discrepancy in the perfect freezing regime]\label{propdisc}
	Draw $\{x_j\}_1^n$ randomly from \eqref{eqprob} under the perfect freezing condition \eqref{eq_lowtemp} and fix $\tau>0$, $L> 2$.
	There are positive
	constants $C$ (universal) and $D=D(c_0,\tau)$ such that  the following assertions hold:	
	\begin{enumerate}[label=(\roman*),ref=(\roman*), itemsep=1em]
		\item (Bulk) Given any $\delta>0$, there is an $n_0=n_0(\delta,L,\tau,c_0)\in\N$ such that for every $n\ge n_0$
		\begin{align}\label{eq_ppp1}
			\mathbf{P}_n^{\beta_n}\left(\left\{
			\sup_{x \in [-2+\delta,2-\delta]}
			\left|\#\big[B(x,L/n)\big]-2\sigma(x)L\right|>D\log^{2} L\right\}\right)\le C(n^{-\tau}+ e^{- \beta_n n /C}).
		\end{align}
		\item (Boundary) There is an $n_0=n_0(L,\tau,c_0)\in\N$ such that for every $n\ge n_0$ and $x= \pm 2$,
		\begin{align}\label{eq_ppp2}
			\mathbf{P}_n^{\beta_n}\left(\left\{ \left|\#\big[B(x,(L/n)^{\frac 23})\big]-\tfrac{2}{3\pi}L\right|> D L^{\frac 45}\log^{\frac 15} L\right\}\right)\le C(n^{-\tau}+ e^{- \beta_n n /C}).
		\end{align}
	\end{enumerate}
\end{theorem}

Fixing a value $\tau>1$ and taking limit for $n\to \infty$ we get the following corollary.
\begin{corollary}\label{cordisc}
	Let $L> 2$. Under the perfect freezing condition \eqref{eq_lowtemp} there are positive constants $C$ (universal) and $D=D(c_0)$ such that almost surely the following assertions simultaneously hold:	
	\begin{enumerate}[label=(\roman*),ref=(\roman*), itemsep=1em]
		\item (Bulk) For every $y_n\to x\in(-2,2)$,
		\[\limsup_{n\to\infty} \Big|\#\big[B(y_n,L/n)\big]-2\sigma(x)L\Big|\le D\log^{2} L.\]
		\item (Boundary) For $x= \pm 2$,
		\[\limsup_{n\to\infty} \Big|\#\big[B(x,(L/n)^{\frac 23})\big]-\tfrac{2}{3\pi}L\Big|\le DL^{\frac 45}\log^{\frac 15} L.\]
	\end{enumerate}
\end{corollary}

We stress two important features of  Theorem \ref{propdisc}.
Firstly, in \eqref{eq_ppp1} the number of points is controlled simultaneously for every interval of length $L/n$ away from the boundary. Secondly, Theorem \ref{propdisc} concerns the critical scales
($n^{-1}$ for the bulk and $n^{-\frac 23}$ for the boundary) while several related results for fixed $\beta<\infty$ involve mesoscopic scales \cite{BaSu,BoErYa,SoWo,Wo}.
However, it is pertinent to point to \cite[Theorem~1.1]{HoPa} which implies that
for fixed $\beta <\infty$, all fixed $x_0\in(-2,2)$ and all $\varepsilon>0$, we have
\begin{align}\label{sos}\lim_{L\to \infty} \lim_{n\to \infty}
	\mathbf{P}_n^\beta \left(\left\{\left|\frac{\left|\#[B(x_0,L/n)]-2\sigma(x_0)L\right|}{\log L} -\frac{2}{\sqrt{\beta}\pi}\right|>\varepsilon \right\}\right)=0,\end{align}
see also \cite{Zh}.

There are several important differences between the convergence in \eqref{sos} and the result in \eqref{eq_ppp1}. Firstly, \eqref{eq_ppp1} involves taking a supremum over $x\in [-2+\delta,2-\delta]$ and this is not done in \eqref{sos}.
Secondly,
the estimate \eqref{eq_ppp1} holds for fixed finite $L$, not just asymptotically as $L\to\infty$. Moreover, it follows from
Theorem~\ref{necbl} that the low probability event in
\eqref{sos} does occur almost surely for infinitely many values of $n$, and, indeed, the $\lim_n$ on \eqref{sos} is not $0$ for fixed $L$.

\subsection{Comments} We now comment on various more or less related issues, provide some additional details on how our strategy compares with earlier approaches in dimension two, and so on.

In recent years, much work has been done in the study of Gaussian $\beta$-ensembles for non-classical $\beta$ (see for example \cite{BaSu,LP22,LeRi,SoWo,Tr,VaVi}), to a great extent building on the tridiagonal matrix model from \cite{DuEd}. Universality results for $\beta$-ensembles with general potentials have been derived in \cite{BoErYa,BoErYa2,ErYa}. The distribution of Fekete points ($\beta=\infty$) in the Gaussian potential is well understood due to their characterization as zeros of Hermite polynomials \cite{DKMLV,KrML,MNT} (besides proving they are nodes of the Gaussian quadrature). The separation and discrepancy results that we present here show a similar pattern for finite but $n$-dependent $\beta \to \infty$. (The opposite end, the high-temperature crossover $\beta \to 0$ is studied in \cite{NaTr,Tr}.) We point out that for traditional $\beta$-ensembles, starting with the well known Wigner surmise, various different notions of spacings have been introduced and studied in great detail. For more about this we refer to \cite{AGK22,AMP,BoErYa2,De99,DH90,ErYa,Fo,GMW,Meh,SV15,Sh14,TiRiKa} and the references there. We also mention that the $\text{Sine}_\beta$ process was shown in \cite{Le16} to converge as $\beta \to \infty$ to a randomly translated lattice.

Our present work uses ideas developed in dimension two to deal with the distribution of Fekete points \cite{AmOC}. The strategy in \cite{AmOC} is modeled on important work due to Beurling and Landau in the context of Paley-Wiener spaces, which
has been adapted to prove equidistribution in several other settings, see \cite{MR222554,MR2929058} as well as \cite{MR2006561,Seip}, for example. The book \cite{BHS} covers many related questions about minimum energy points on manifolds.

The technique from \cite{AmOC} was applied in a random setting of low temperature Coulomb gas in the papers \cite{Am2,AmRom}; in particular the idea of defining Beurling-Landau densities of random samples is from \cite{AmRom}. The idea of using random Lagrange functions was inspired by an earlier construction in complex line bundles
\cite{MR3831027}.
We refer to \cite{AMP,AmRom,AS21,CSA,F22,Fo,Lew22} and the references for many more related results and other approaches in dimension two or higher.

At the surface, our approach for proving Theorem \ref{sufbl} is similar to \cite{AmRom}: we first prove sampling and interpolation inequalities by studying random Lagrange polynomials, and then leverage these to prove discrepancy estimates elaborating on a technique that can be traced back to Landau's work on bandlimited functions \cite{MR222554}. The particulars of the one and two dimensional cases are however quite different, for several reasons. In dimension two, the same microscopic scale $1/\sqrt{n}$ typically holds globally throughout the droplet, whereas in dimension one the microscopic scale typically varies from $1/n$ in the bulk to $n^{-\frac 2 3}$ at the boundary, which requires different techniques. To deal with this, we use Bernstein, Nikolskii and restriction inequalities  from the theory of orthogonal polynomials \cite{LeLu2}, as well as Plancherel-Rotach type asymptotics for Hermite polynomials from \cite{DKMLV,Sz}.
Another difference stems from the fact that \cite{AmRom} concerns weighted polynomials in one \emph{complex} variable which are estimated using local inequalities related to Cauchy's integral formula. These are not available in the present real-variable setting, and we resort to some rather subtle estimates for weighted Gaussian polynomials on the real line (such as Lemma~\ref{lemmabessel} below) that balance local and global quantities.

For the implementation of Landau's technique in the present setting we benefit from recent spectral estimates for certain Toeplitz operators
\cite{BoJaKa,israel15,KaRoDa,KaZhWaRoDa,Os} that are finer and less asymptotic than the ones in the classical work of Landau and Widom \cite{MR593228, MR169015}.

We expect that the results which are here derived for the Gaussian potential should be universal for a large class of potentials $V(x)$ on the real line.
In order to generalize to this case using our present techniques, we would require off-diagonal estimates for the kernel $K_n(x,y)$ similar to Lemma \ref{leker} below. These should follow from uniform  asymptotics for orthogonal polynomials with respect to general exponential weights (see for example \cite{DeKrMLVeZh2,MLMi}). One would also need a general version of Theorem \ref{lemmaprob}, which seems unknown (cf. \cite{ErXu}).

Another kind of natural question would be to study the perfect freezing problem in the regime between dimensions one and two. The model case
is here the collapsing elliptic $\beta$-ensemble from the paper \cite{AB21}.

\subsection{Organization}
The article is organized as follows.
\newline In Section~\ref{pnecbl} we provide necessary and sufficient conditions for equidistribution. We show how Theorem \ref{sufbl} follows from Theorem~\ref{propdisc} and prove Theorem \ref{necbl}.
\newline In Section \ref{secpre} we recall some background results. Lemma~\ref{remsep} provides diagonal estimates for the reproducing kernel. Theorem \ref{theop}, Corollary \ref{coroniko} and Theorem \ref{theores} are Bernstein, Nikolskii and restriction inequalities for weighted polynomials.
\newline In Section~\ref{secsep} we prove almost-sure uniform bounds of Lagrange functions, Theorem \ref{tub}, and our separation result, Theorem~\ref{sepa}.
\newline In Section~\ref{secsam} we derive sampling and interpolation for weighted polynomials at the configuration points, Theorem~\ref{theosam}. 
\newline In Section~\ref{secdis} we prove the main result on discrepancy, Theorem~\ref{propdisc}.
\newline In Section~\ref{seclem} we address off-diagonal estimates for the reproducing kernel (Lemma~\ref{leker}) and the technical Lemma~\ref{lemtech} that is key to the proof of Theorem~\ref{theosam}.

\subsection{Notation}
For a function $f$, we will use $f^{-1}$ to denote $1/f$ (rather than the inverse function). Many quantities depend on the sample-size $n$, but we often drop this dependence in the notation. For expressions $a$ and $b$ depending on $n,x,y,\varepsilon$, we write $a\lesssim b$ and $a \simeq b$ to indicate there is a constant $C>0$ independent of $n,x,y,\varepsilon$ (but possibly depending on other parameters) such that $a\le C b$ and that $a\lesssim b \lesssim a$ respectively. Finally, an averaged integral is denoted by
$$\fint_E f := \frac{1}{|E|}\int_E f(x)\, dx.$$

\section{Equidistribution} \label{pnecbl}

In this section we deduce sufficiency and necessity for equidistribution in terms of the condition $\beta_n\ge c_0\log n$, namely, Theorem~\ref{sufbl} and Theorem~\ref{necbl}. We start with Theorem~\ref{sufbl} and assume for now that Theorem~\ref{propdisc} is proven.

\begin{proof}[Proof of Theorem \ref{sufbl} assuming Theorem~\ref{propdisc}]
	For the first statement \ref{bulk}, we use \eqref{eq_ppp1} in Theorem~\ref{propdisc} with $L>2$, $\tau=2$ and $\delta=1/k$ with $k\in\N$. By the Borel-Cantelli lemma,
	\[\limsup_{n\to \infty} \sup_{x \in [-2+1/k,2-1/k]}
	\left|\#\big[B(x,L/n)\big]-2\sigma(x)L\right|\le D\log^{2} L \qquad\text{almost surely.}\]
	Since $k\in \N$ was arbitrary, almost surely the last inequality holds for every $k\in\N$ simultaneously. In particular,
	\[ \sup_{\delta>0}\limsup_{n\to\infty}
	\sup_{x \in [-2+\delta,2-\delta]}
	\left|\#\big[B(x,L/n)\big]-2\sigma(x)L\right|\le D\log^{2} L \qquad\text{almost surely.}\]
	Dividing the last expresion by $2L$ and taking $L\to\infty$,
	we get \ref{bulk}.
	The second and third statements follow analogously by a Borel-Cantelli argument applied to \eqref{eq_ppp2} and \eqref{erl} respectively.
\end{proof}

Regarding necessity, we begin with an auxiliary lemma that follows from~\cite{LeRi}.

\begin{lemma} \label{conju2} Let $\beta_n$ be a sequence of inverse temperatures that is not in the perfect freezing regime, that is,
	\begin{equation}\label{lid}\liminf_{n\to\infty}\frac {\beta_n}{\log n}=0.\end{equation}
	Then
	\begin{align}
		\label{contraba}\limsup_{n\to\infty} n^{\frac 23}(2-\max_{1\le j\le n}x_j)&=\infty\qquad \text{ almost surely,}
		\intertext{and}
		\label{contraba2}\limsup_{n\to\infty} n^{\frac 23}(\max_{1\le j\le n}x_j-2)&=\infty\qquad \text{ almost surely.}
	\end{align}	
\end{lemma}

\begin{proof}
	
	By \cite[Theorem~4]{LeRi}, there is a constant $C>1$ such that for every $k\in\N$, and all sufficiently large $n$ (depending on $k$):
	\[\mathbf{P}_n^{\beta_n}\left(\left\{\max_{1\le j\le n} x_j \le 2-kn^{-\frac 23}\right\}\right)\ge  C^{-\beta_n k^3}.\]
	
	Since, by \eqref{lid}, $\sum_n C^{-\beta_n k^3}$ diverges, the second Borel-Cantelli lemma gives, almost surely,
	\[\max_{1\le j\le n} x_j \le 2-kn^{-\frac 23},\]
	for infinitely many $n$.
	In other words,
	\[\limsup_n n^{\frac 23}\big(2-\max_{1\le j \le n} x_j\big)\ge k \qquad \text{almost surely.}\]
	Since $k$ was arbitrary, we get \eqref{contraba}.
	
	The proof of \eqref{contraba2} is identical, this time using another part of \cite[Theorem~4]{LeRi}: there is a constant $C>1$ such that for every $k\in\N$, and $n$ sufficiently large (depending on $k$):
	\[\mathbf{P}_n^{\beta_n}\left(\left\{\max_{1\le j\le n} x_j \ge 2+kn^{-\frac 23}\right\}\right)\ge  C^{-\beta_n k^{3/2}}.\qedhere\]
\end{proof}

Now we turn to the proof of Theorem \ref{necbl}.

\begin{proof}[Proof of Theorem \ref{necbl}]
	
	We start with the bulk part \ref{ZW}.
	Assume that $\beta>0$ is \textit{independent of $n$} and fix a point $x_0$ with $-2<x_0<2$.
	We shall use the fact that the microscopic limit of the processes $\{x_j\}_1^n$ at $x_0$ is the $\text{Sine}_\beta$ process
	\cite{VaVi}, which, by \cite{HoVa}, has an arbitrary large number of points within a fixed interval with positive probability.
	
	Fix $L>0$, $k\in \N$ and denote the number of particles of the $\text{Sine}_\beta$ in an interval $[0,\lambda]$ by $N_\beta(\lambda)$. It follows from  \cite[Theorem~1]{VaVi} and \cite[Theorem~1]{HoVa} that
	\[\lim_{n\to\infty} \mathbf{P}_n^{\beta}\left(\left\{\#[B(x_0,L/n)]\ge k\right\}\right)=\mathbf{P}_{\text{Sine}_\beta}\Big(\Big\{N_\beta\big(2L \sqrt{4-x_0^2}\big)\ge k\Big\}\Big)>0.  \]
	Since samples for different values of $n$ are independent, by the Borel-Cantelli lemma for every $k\in \N$,
	\begin{align*}
		\limsup_{n\to \infty} \#[B(x_0,L/n)]\ge k \qquad  \text{almost surely.}
	\end{align*}
	Hence,
	\begin{equation}
		\label{rbc}\limsup_{n\to \infty} \#[B(x_0,L/n)]= \infty \qquad  \text{almost surely.}
	\end{equation}
	We now take $L_j=1/j$ with $j\in\N$ and use a monotonicity argument, to conclude that, almost surely, \eqref{rbc} holds simultaneously for every $L>0$, which proves the first equality in \ref{ZW}.
	
	Similarly, from  \cite[Theorem~5]{VaVi} for every $L>0$,
	\[\lim_{n\to\infty} \mathbf{P}_n^{\beta}\left(\left\{\#[B(x_0,L/n)]=0 \right\}\right)=\mathbf{P}_{\text{Sine}_\beta}\Big(\Big\{N_\beta\big(2L \sqrt{4-x_0^2}\big)=0 \Big\}\Big)>0,\]
	Analogous Borel-Cantelli and monotonicity arguments show that almost surely for every $L>0$,
	\begin{align*}
		\liminf_{n\to \infty} \#[B(x_0,L/n)]=0,
	\end{align*}
	which completes the proof of \ref{ZW}.
	
	Regarding the edge part \ref{ZE}, suppose that the condition $\beta_n\gtrsim \log n$ fails, that is,
	\begin{equation*}\liminf_{n\to\infty}\frac {\beta_n}{\log n}=0.\end{equation*}
	
	By \eqref{contraba} in Lemma \ref{conju2} we conclude that almost surely, for each $L$ and each $n_0$ there is $n\ge n_0$ such that $\#[B(2,(L/n)^{\frac 23})]=0$.
	Hence almost surely for every $L>0$, $$\liminf_{n\to\infty}\#[B(2,(L/n)^{\frac 23})]=0.$$
	This proves \ref{ZE} for $x=2$. The case $x=-2$ is completely symmetric.
	
	The proof of \ref{ZZ} is analogous to \ref{ZE}. From \eqref{contraba} in Lemma \ref{conju2}, almost surely for each $L$ and each $n_0$ there is $n\ge n_0$ such that $\#\big[[2+Ln^{-\frac 23},\infty)\big]\ge 1$. In other words,
	\[\liminf_{n\to\infty}\#\big[[2+Ln^{-\frac 23},\infty)\big]\ge 1.\qquad \text{almost surely.}\qedhere\]
\end{proof}

\section{Further auxiliary facts}\label{secpre}
We recall some well-known estimates for the reproducing kernel $K_n$ of the space $\mathcal{W}_n$ of weighted polynomials, given by
\[K_n(x,y)=\sum_{k=0}^{n-1}p_k(x)p_k(y),\]
where the functions $p_k$ are weighted and suitably scaled Hermite polynomials (see \eqref{eqherm} for the definition). 
Notice that the kernel $K_n$ is real-valued and symmetric: $K_n(x,y)=K_n(y,x)$.

It is convenient to define the following function, which is $n$ times a suitable modification of the regularized semicircle density defined in \eqref{eq_psi}:
\begin{equation}\label{ex_psi}\psi_n(x):=n\cdot \sqrt{\max\{1-|x|/2,n^{-\frac 23}\}}\simeq n\sigma_n(x).
\end{equation}
This function has an important relationship with the
diagonal values $K_n(x,x)$, as the following result from \cite[Theorem~1.1]{LeLu3} shows.

\begin{lemma}\label{lspinoff}
	For each $L>0$, there are constants $C_1,C_2>0$ such that for every $|x|\le 2+Ln^{-\frac 23}$ we have
	\begin{align*}
		C_1 \psi_n(x) \le K_n(x,x) \le C_2 \psi_n(x).
	\end{align*}
	Moreover, the right-hand side inequality holds for every $x\in \R$.
\end{lemma}

Note that in the perfect freezing regime, thanks to the localization in Theorem \ref{lemmaprob}, the functions $\psi_n(x)$ and $K_n(x,x)$ are equivalent on the range
of almost every $\{x_j\}_1^n$ when $n$ is large enough.

The following lemma provides some additional diagonal estimates for the kernel complementing Lemma~\ref{lspinoff}. (Here and in what follows $\sigma(x)$ is always the semicircle law \eqref{eq_sc}.)

\begin{lemma}
	\label{remsep}
	The following estimates hold:
	\begin{enumerate}[label=(\roman*),ref=(\roman*)]
		\item \label{remit2} For $L>0$, every $0<\varepsilon<1$ and $|x|\le 2+Ln^{-\frac 23}$ we have
		\[K_{n}(x,x)\lesssim \frac{1}{\varepsilon} K_{n\varepsilon}(x,x).\]
		\item \label{remit3} For $\alpha\in (0,2)$ and uniformly on $|x|<\alpha$,
		\[ \frac{K_n(x,x) }{n}=\sigma(x)\cdot (1+\so(1)).\]
	\end{enumerate}
\end{lemma}
\begin{proof}
	The second statement \ref{remit3} is a rescaled version of \cite[Theorem~1.25]{LeLu2}.
	
	Regarding \ref{remit2}, from Lemma~\ref{lspinoff} we have if $|x|\le 2-2n^{-\frac 23}$ that
	\begin{align*}K_n(x,x) &\lesssim n \sqrt{1-|x|/2} \le \frac{1}{\varepsilon} n\varepsilon \sqrt{\max\{1-|x|/2,(n\varepsilon)^{-2/3}\}} \\
		&\lesssim \frac{1}{\varepsilon} K_{n\varepsilon}(x,x),\end{align*}
	while if $2-2n^{-\frac 23}\le |x|\le 2+Ln^{-\frac 23}$,
	\begin{align*}K_n(x,x) &\lesssim n^{1/3}\le \frac{1}{\varepsilon} (n\varepsilon)^{1/3}\le\frac{1}{\varepsilon} n\varepsilon \sqrt{\max\{1-|x|/2,(n\varepsilon)^{-2/3}\}}\\
		&\lesssim \frac{1}{\varepsilon} K_{n\varepsilon}(x,x). \qedhere\end{align*}
\end{proof}

Now we turn to three important tools in the theory of orthogonal polynomials: Bernstein, Nikolskii and restriction inequalities.

\begin{theorem}[Bernstein-type inequality]\label{theop}
	For every $1\le p \le \infty$ there is a constant $C>0$ such that for every $f\in\ww_n$,
	\begin{align*}
		\|f'/\psi_n\|_p \le C \|f\|_{p},
	\end{align*}
	where $\psi_n(x)$ is given in \eqref{ex_psi}.
\end{theorem}

\begin{proof}
	The result follows from rescaling Bernstein's inequalities for the weight $e^{-x^2}$, namely, \cite[Theorem~1.3]{LeLu} when $p=\infty$ and \cite[Theorem~1.1]{LeLu4} when $1\le p<\infty$.
\end{proof}

\begin{corollary}[Nikolskii-type inequality]\label{coroniko}
	There is a constant $C>0$ such that for every $1\le p\le \infty$:
	\begin{align*}
		\|f\|_\infty \le C n^{\frac 1p} \|f\|_{p}, \quad f \in \ww_n.
	\end{align*}
\end{corollary}
\begin{proof}
	Since  $\psi_n\le n$, as a particular case of Theorem \ref{theop}, we have that $\|f'\|_\infty\le Cn \|f\|_\infty$.
	We now leverage this bound to go from the $\infty$-norm to the $p$-norm.
	Suppose that $|f|$ attains its maximum at $s\in\R$. For every $t\in\R$ such that $|s-t|\le 1/(2Cn)$ we have
	\begin{align*}|f(t)|&=|f(s)+f(t)-f(s)|\\
		&\ge |f(s)|-\|f'\|_\infty |t-s|\\
		&\ge \|f\|_\infty - \frac{Cn}{2Cn}\|f\|_\infty=\frac{1}{2}\|f\|_\infty.\end{align*}
	So we get
	\begin{gather*}
		\|f\|_\infty\le 2 (Cn)^{\frac 1p}\Big(\int_{s-1/(2Cn)}^{s+1/(2Cn)} |f(t)|^{p}\, dt\Big)^{\frac 1p}\le 2 C' n^{\frac 1p} \|f\|_{p}. \qedhere
	\end{gather*}
\end{proof}

The following restriction theorem simplifies many considerations near the boundary. (This is a point where the parallel two-dimensional theory is more subtle, see e.g. the estimates on weighted polynomials in \cite[Section III.6]{SaTo} as well as \cite{AmRom}.)

\begin{theorem}[Restricted range inequality]\label{theores}
	For every $f\in \ww_n$ we have
	\[\|f\|_2^2 \leq 2 \int_{-2}^{\,2}|f|^2.
	\]
\end{theorem}

\begin{proof}
	Rescaling \cite[Theorem~4.1]{LeLu2} applied to the weight $e^{-x^2}$ gives the result.
\end{proof}
\section{Uniform separation}
\label{secsep}
In this section we prove Theorems~\ref{tub} and \ref{sepa}.
We let $\{x_j\}_{j=1}^n$ be randomly drawn from \eqref{eqprob} and start by providing some estimates for the Lagrange polynomials $\ell_j$ defined in \eqref{eq_lp}. The following result is analogous to \cite[Lemma~2.5]{Am} and \cite[Lemma~3.1]{AmRom}. We provide a proof for convenience of the reader. Similar identities were used in \cite{MR3831027}.
\begin{lemma}\label{lemmavar}
	For $\Lambda\subseteq \R$ a measurable set, $\beta>0$ and $1\le j\le n$ we have
	\[\mathbf{E}_n^\beta\Big[\int_\R 1_{\Lambda}({x}_j)|\ell_j({x})|^{\beta} \, d{x} \Big]= |\Lambda|.\]
\end{lemma}
\begin{proof}
	Without loss of generality assume $j=1$.
	Notice that the density in \eqref{eqprob} obeys
	\[|\ell_1({x})|^{\beta}   p_n^{\beta}(x_1,\ldots,x_n)=  p_n^{\beta}(x,x_2,\ldots,x_n).\]
	By Fubini's theorem,
	\begin{align*}
		\mathbf{E}_n^\beta\Big[\int_\R 1_{\Lambda}({x}_1)|\ell_1({x})|^{\beta} \, d{x} \Big] &=
		\int_{\R^n} \int_\R 1_{\Lambda}({x}_1)|\ell_1({x})|^{\beta}   p_n^{\beta}(x_1,\ldots,x_n) \, d{x} dx_1\ldots dx_n
		\\ &=
		\int_\R 1_{\Lambda}({x}_j) d{x_j}  	\int_{\R^n} p_n^{\beta}(x,x_2,\ldots,x_n) \, dx dx_2,\ldots dx_n
		= |\Lambda|. \qedhere
	\end{align*}
\end{proof}
Second, we estimate the $\beta$-norm of the Lagrange polynomials.
\begin{lemma}\label{lemmabound}
	There is a universal constant $C>0$ such that for every $\tau>0$
	and every $\beta\ge 1$ we have
	\[\mathbf{P}_n^\beta\left(\left\{\max_{1\le j \le n} \|\ell_j\|_{\beta}^{\beta} > n^{1+\tau}\right\} \right)<C (n^{-\tau}+e^{-n \beta/C}).\]
\end{lemma}
\begin{proof}
	Let $\Lambda=[-3,3]$ (any neighborhood of the droplet would work). By Theorem~\ref{lemmaprob},
	\[\mathbf{P}_n^\beta\left(\left\{\{{x}_j\}_1^n\nsubseteq \Lambda\right\}\right)=\mathbf{P}_n^\beta\left(\left\{\max_{1\le j\le n}|{x}_j|>2+1\right\}\right)\le C e^{-n \beta /C}.\]
	From this, Chebyshev's inequality and Lemma~\ref{lemmavar},
	\begin{align*}
		\mathbf{P}_n^\beta\left(\left\{\max_{1\le j \le n} \|\ell_j\|_{\beta}^{\beta} > n^{1+\tau} \right\}\right)
		&\le \mathbf{P}_n^\beta\left(\left\{\max_{1\le j \le n} \|\ell_j\|_{\beta}^{\beta} > n^{1+\tau}, \{{x}_j\}_1^n\subseteq \Lambda\right\}\right)\\
		&\quad  +
		\mathbf{P}_n^\beta\left(\left\{\{{x}_j\}_1^n\nsubseteq \Lambda\right\}\right)
		\\ &\le \mathbf{P}_n^\beta\left(\left\{\max_{1\le j \le n} 1_{\Lambda}({x}_j) \|\ell_j\|_{\beta}^{\beta} > n^{1+\tau}\right\} \right) + C e^{- n \beta /C}
		\\ &\le n \mathbf{P}_n^\beta\left(\left\{ 1_{\Lambda}({x}_1) \|\ell_1\|_{\beta}^{\beta} > n^{1+\tau}\right\} \right) + C e^{- n \beta /C}
		\\ &\le |\Lambda|n^{-\tau}+C e^{- n \beta /C}\le C'(n^{-\tau}+ e^{- n \beta /C}). \qedhere
	\end{align*}
\end{proof}

With this at hand we can prove Theorem~\ref{tub}.
\begin{proof}[Proof of Theorem~\ref{tub}]
	By Corollary \ref{coroniko}, for every $1\le j\le n$,
	\[\|\ell_j\|_{\infty}\le C n^{\frac 1{\beta_{n}}}\, \|\ell_j\|_{\beta_n}\]
	Taking $\tau=2$ in Lemma \ref{lemmabound} we have that
	\[\max_{1\le j \le n} \|\ell_j\|_{\infty}\le C n^{\frac 1{\beta_{n}}} \max_{1\le j \le n} \|\ell_j\|_{\beta_n}\le C  n^{\frac 4{\beta_{n}}}\le C e^{\frac 4{c_0}},\]
	with probability $\geq 1- C (n^{-2}+ e^{- n \beta_n /C})$. The desired conclusion thus follows from the Borel-Cantelli lemma applied to the measure $\mathbf{P}=\prod \pp$.
\end{proof}

Finally, the following observation will be helpful. 

\begin{remark}
	\label{lemma_a}
	Let $a \in \R$, $r>0$ and $x\in B(a,r/\psi_n(a))$, where $\psi_n$ is the function \eqref{ex_psi}. Then
	\[\psi_n(x) \leq \psi_n(a)\cdot \sqrt{1+\frac{r}{2}} .
	\]

	Indeed, the function $\psi_n^2$ is Lipschitz with constant $n^2/2$, since it is piece-wise linear with slopes equal to $0$ or $\pm n^2/2$.
	Hence,
	\begin{align*}
		\psi_n({x})^2&=\psi_n({x})^2-\psi_n(a)^2+\psi_n(a)^2
		\le \frac{n^2}{2}|{x}-a|+\psi_n(a)^2
		\\ & \le \frac{n^2 r}{2 \psi_n(a)}+\psi_n(a)^2
		\le  \frac{r}{2}\psi_n(a)^2+\psi_n(a)^2,
	\end{align*}
	where we used that $n^{2/3}\le \psi_n$. 
\end{remark}

We are now in position to prove the main result on separation.
\begin{proof}[Proof of Theorem \ref{sepa}] Consider a random sample $\{x_j\}_1^n$ from $\P_n^\beta$.

	Choose points $x_j$ and $x_k$ with $j\neq k$ so that the spacing $n\sigma_n({x}_j)\cdot |{x}_j-{x}_k|$ is minimal. Without loss of generality we assume ${x}_j\neq {x}_k$ since this occurs almost surely. We want to show that $s_n=n\sigma_n({x}_j)\cdot |{x}_j-{x}_k|$ is large with high probability. If $s_n\geq 1$, we are done. Assume that $s_n\leq 1$. From \eqref{ex_psi}, there is a universal constant $r>0$ such that
	\[\psi_n\le r n\sigma_n.\]
	Then, for any real ${x}$ between ${x}_j$ and ${x}_k$,
	\begin{align*}
		|x-x_j|\le \frac{1}{n\sigma_n(x_j)} \le \frac{1}{\psi_n(x_j)},
	\end{align*}
	and by Remark \ref{lemma_a},
	\begin{align*}
		\psi_n({x}) \leq \psi_n({x}_j)\sqrt{1+\tfrac{r}{2}}.
	\end{align*}
	By Theorem~\ref{theop} we get
	\begin{align}
		\label{eq_lagr}
		1 &=|\ell_j({x}_j)-\ell_j({x}_k)|=\Big|\int_{{x}_j}^{{x}_k}\ell_j'(t)\psi_n^{-1}(t)\psi_n(t)\, dt\Big|
		\\ &\le \sqrt{1+\tfrac{r}{2}}\, \|\ell_j'\psi_n^{-1}\|_\infty\,  \psi_n({x}_j)\cdot |{x}_j-{x}_k|\\
		&\le C \|\ell_j \|_{\infty}\cdot \mathrm{s}_n. \notag
	\end{align}
	By Corollary \ref{coroniko} and Lemma \ref{lemmabound}, we obtain with probability at least $$1- C'' (n^{-\tau}+ e^{- n \beta /C})$$ that
	\[1\le C \|\ell_j \|_{\infty}\cdot \mathrm{s}_n\le C' n^{\frac 1\beta} \|\ell_j \|_{\beta}\cdot \mathrm{s}_n \le C' n^{\frac {2+\tau}{\beta}}\cdot \mathrm{s}_n.\]
	This proves \eqref{eq_sepaa}. Finally \eqref{eq_sepa}, follows from \eqref{eq_lagr} and Theorem~\ref{tub}.
\end{proof}

\section{Sampling and interpolation}\label{secsam}
In this section we prove Theorem~\ref{theosam} on almost-sure sampling and interpolation.
We start by showing that a function in $\ww_n$ can be controlled by a local average and a global error term with an arbitrary trade-off parameter.
This will play the role of the Bernstein-type inequalities from the two dimensional setting. In addition, we prove Theorem \ref{tub} on uniform upper bounds of the Lagrange functions $\ell_j$.

\begin{lemma}[Bessel-type bound]\label{lemmabessel}
	Let $$\{x_j\}_1^n \subseteq[-2-Mn^{-2/3},2+Mn^{2/3}]$$ be a real configuration which is uniformly separated, i.e.,
	$$\mathrm{s}_n\geq s_0>0,$$
	where $\mathrm{s}_n$ is the normalized separation \eqref{eq_sep}.
	
	Then there exist constants $s,C>0$ depending only on $M$ and $s_0$ such that for every $\alpha\in (0,1)$, every $f\in \ww_n$ and every $J\subseteq \{1,\ldots,n\}$:
	\begin{align}
		\label{bessel}
		\sum_{j\in J} \frac{|f({x}_j)|^2}{K_n({x}_j,{x}_j)} \le \frac{1}{s\alpha} \sum_{j\in J}\int_{I_j}|f|^2 +s\alpha C \|f\|_2^2,
	\end{align}
	where $\{I_j\}_{j=1}^n$ are disjoint intervals given by
	\[I_j={x}_j+\frac{s}{K_n({x}_j,{x}_j)}[-1,1].\]
	In particular, there is a constant $C'>0$ such that
	\begin{align}
		\label{bessel2}
		\sum_{j=1}^n \frac{|f({x}_j)|^2}{K_n({x}_j,{x}_j)} \le C' \|f\|_2^2.
	\end{align}
\end{lemma}

\begin{proof}[Proof of Lemma~\ref{lemmabessel}]
	Since $s_n\geq s_0$, by \eqref{ex_psi} and Lemma \ref{lspinoff} we can choose $0<s\le 1$ sufficiently small so that the intervals $I_j$ are indeed disjoint. For $\alpha \in (0,1)$ and $j \in J$ define $A_j=x_j+\alpha(I_j-x_j)$.
	Since
	\[\fint_{A_j}|f|^2=\frac{1}{\alpha |I_j|} \int _{A_j}|f|^2
	\le \frac{1}{\alpha |I_j|} \int _{I_j}|f|^2 =\frac{1}{\alpha } \fint _{I_j}|f|^2 \]
	there exists ${y}_j\in A_j$ such that
	\[|f({y}_j)|^2\le \frac{1}{\alpha} \fint_{I_j} |f|^2.\]
	As a consequence,
	\begin{align}
		\label{bound1}
		\frac{|f({x}_j)|^2}{K_n({x}_j,{x}_j)}&\le 2\frac{|f({y}_j)|^2}{K_n({x}_j,{x}_j)}+2\frac{|f({x}_j)-f({y}_j)|^2} {K_n({x}_j,{x}_j)} \notag
		\\ &\le \frac{2}{\alpha K_n({x}_j,{x}_j)} \fint_{I_j} |f|^2+2\frac{|f({x}_j)-f({y}_j)|^2}{K_n({x}_j,{x}_j)} \notag
		\\ &=\frac{1}{s\alpha } \int_{I_j} |f|^2+2\frac{|f({x}_j)-f({y}_j)|^2}{K_n({x}_j,{x}_j)}.
	\end{align}
	Note that
	\begin{align}
		\label{eq_der}
		\frac{|f({x}_j)-f({y}_j)|^2}{K_n({x}_j,{x}_j)}\le \frac{1}{K_n({x}_j,{x}_j)} \Big(\int_{A_j} |f'|\Big)^2
		\le \frac{2s\alpha}{K_n({x}_j,{x}_j)^2}\int_{I_j} |f'|^2.
	\end{align}
	
	Now recall the function $\psi_n(x)$ defined in \eqref{ex_psi}.
	By Lemma \ref{lspinoff} there is a constant $C_1$ such that $C_1 \psi_n({x_j})\le K_n({x}_j,{x}_j)$ and in particular assuming $s\le C_1$,
	\[\psi_n(x_j)\cdot |x-x_j|\le \psi_n(x_j)\frac{s}{K_n({x}_j,{x}_j)}
	\le \frac{C_1\psi_n(x_j)}{K_n({x}_j,{x}_j)} \le 1, \qquad x\in I_j.\]
	So, by Remark \ref{lemma_a},
	\[\frac{1}{K_n({x}_j,{x}_j)^2}\le \frac{1}{C_1^2\psi_n(x_j)^2}\le \frac{3}{2 C_1^2\psi_n(x)^2}, \qquad x\in I_j.\]
	Combining this with \eqref{eq_der},
	\begin{align}
		\label{bound2}
		\frac{|f({x}_j)-f({y}_j)|^2}{K_n({x}_j,{x}_j)}\le s\alpha C \int_{I_j} |f'|^2\psi_n^{-2}.
	\end{align}
	From \eqref{bound1} and \eqref{bound2} we obtain
	\[\sum_{j=1}^n \frac{|f({x}_j)|^2}{K_n({x}_j,{x}_j)} \le \frac{1}{s\alpha}\sum_{j\in J}\int_{I_j}|f|^2 + s \alpha C \|f' \psi_n^{-1}\|_2^2.\]
	The conclusion follows from Theorem \ref{theop} with $p=2$.
\end{proof}

Whereas we will only need \eqref{bessel2} to prove the sampling and interpolation theorem, the proof of the discrepancy estimates below relies on the most subtle estimate \eqref{bessel}. The second ingredient towards sampling and interpolation is the following reproducing kernel bound whose proof is postponed to Section~\ref{seclem}.

\begin{lemma}
	\label{lemtech}
	The following statements hold.
	\begin{enumerate}
		\item Given $M\ge1$ there is a constant $C>0$ such that for every $|y|\le 2+ M n^{-2/3}$:
		\begin{align}
			\label{eqint}
			\int_{-2}^2\frac{K_n(x,y)^2}{K_n(x,x)^3}dx\le C\frac{1}{ K_n(y,y)^2}.
		\end{align}
		\item Given $M\ge1$ there is a constant $C>0$ such that for every $|y|\le 2+ M n^{-2/3}$:
		\begin{align}
			\label{eqint2}
			\int_{-2}^2K_n(x,y)^6 K_n(x,x)dx\le C K_n(y,y)^6.
		\end{align}
	\end{enumerate}
\end{lemma}

With these tools at hand, we are now ready to show that the Gaussian ensemble in the perfect freezing regime solves the sampling and interpolation problems for weighted polynomials.

\begin{proof}[Proof of Theorem \ref{theosam}]
	To avoid some uninteresting technicalities, we assume throughout that $\rho$ in (Interpolation) and (Sampling) is such that $n\rho$ is an integer.
	
	We invoke Theorem~\ref{sepa}, Theorem~\ref{lemmaprob} and Lemma~\ref{lemmabound}. By restricting to an event of probability $\mathbf{P}_n^{\beta_n}$ at least $1- C(n^{-\tau}+ e^{- \beta_n n /C})$ we ensure that \ref{eqwid} and \ref{eqsepar} are satisfied, and that
	$\|\ell_j\|_\beta \leq n^{\frac{1+\tau}{\beta_n}}$.
	Since $n^{\frac{1}{\beta_n}} \asymp 1$, we conclude from Corollary~\ref{coroniko} that
	\begin{align}
		\label{eqlag}
		\max_{1\le j\le n}\|\ell_j\|_\infty \le C \max_{1\le j\le n}\|\ell_j\|_{\beta_n} \le C n^{\frac {1+\tau}{\beta_n}}\le C e^{\frac {1+\tau}{c_0}}.
	\end{align}
	
	\textit{Step 1. (Interpolation).} Let us write $\rho=1+4\varepsilon$. We can assume $n\varepsilon$ is an integer.
	For $(a_j)_{j=1}^n\subseteq \C$
	define $f \in \ww_{n\rho}$ by
	\[f(x) = \sum_{j=1}^n a_j \ell_j({x}) \frac{K_{n\varepsilon}({x}_j,{x})^4}{K_{n\varepsilon}({x}_j,{x}_j)^4}.\]
	It is clear that $f(x_j)=a_j$ for every $1\le j\le n$.
	Regarding the interpolant's 2-norm, by Theorem~\ref{theores} and \eqref{eqlag},
	\begin{align*}
		\int_\R |f|^2
		&\lesssim \int_{-2}^2 |f|^2 \lesssim \int_{-2}^2 \Big|\sum_{j\in I} a_j \ell_j({x}) \frac{K_{n\varepsilon}({x}_j,{x})^4}{K_{n\varepsilon}({x}_j,{x}_j)^4}\Big|^2\, dx  \notag
		\\ &\lesssim
		\int_{-2}^2 \Big(\sum_{j\in I}  \frac{|a_j|}{\sqrt{K_{n}({x}_j,{x_j})}} \frac{|K_{n\varepsilon}({x}_j,{x})|^3 K_{n}({x}_j,{x}_j)}{K_{n\varepsilon}({x}_j,{x}_j)^4}\cdot \frac{|K_{n\varepsilon}({x}_j,{x})|}{\sqrt{K_{n}({x}_j,{x}_j)}}\Big)^2\, dx \notag
		\\ &\le \int_{-2}^2 \Big(\sum_{j\in I} \frac{|a_j|^2}{K_{n}({x}_j,{x_j})} \frac{K_{n\varepsilon}({x}_j,{x})^6 K_{n}({x}_j,{x}_j)^2}{K_{n\varepsilon}({x}_j,{x}_j)^8}\Big) \Big(\sum_{j=1}^n \frac{K_{n\varepsilon}({x}_j,{x})^2}{K_{n}({x}_j,{x}_j)}\Big)\, dx. \notag
		\intertext{So, applying Lemmas \ref{lemmabessel} and \ref{remsep}, and Lemma \ref{lemtech}  for $n\varepsilon$,}
		\int_\R |f|^2 &\lesssim \int_{-2}^2 \Big(\sum_{j\in I} \frac{|a_j|^2}{K_{n}({x}_j,{x_j})} \frac{K_{n\varepsilon}({x}_j,{x})^6K_{n}({x}_j,{x}_j)^2}{K_{n\varepsilon}({x}_j,{x}_j)^8}\Big) \|K_{n\varepsilon}(\cdot,{x})\|_2^2\, dx \notag
		\\ &\lesssim \frac{1}{\varepsilon^2} \sum_{j\in I}  \frac{|a_j|^2}{K_{n}({x}_j,{x}_j)} \int_{-2}^2 \frac{K_{n\varepsilon}({x}_j,{x})^6 K_{n\varepsilon}(x,{x})}{K_{n\varepsilon}({x}_j,{x}_j)^6}\,  dx \notag
		\\ &\lesssim \frac{1}{\varepsilon^2} \sum_{j\in I}  \frac{|a_j|^2}{K_{n}({x}_j,{x}_j)}.
	\end{align*}
	
	\smallskip
	
	\textit{Step 2. (Sampling).} Let us write $\rho=1-2\varepsilon$. We can assume $n\varepsilon$ is an integer. For every $f\in\ww_{n\rho}$ and for every ${x},{y} \in\R$ we have
	\[f({x})K_{n\varepsilon}({x},{y})^2=\sum_{j=1}^n f({x}_j) K_{n\varepsilon}({x}_j,{y})^2 \ell_j({x}).\]
	From Theorem~\ref{theores} and \eqref{eqlag},
	\begin{align*}
		\int_\R |f|^2 &\lesssim \int_{-2}^{2} |f|^2 \notag
		= \int_{-2}^{2} \frac{\Big|\sum_{j=1}^n f(x_j)  K_{n\varepsilon}({x}_j,{x})^2\ell_j({x})\Big|^2}{K_{n\varepsilon}(x,x)^4}\, dx \notag
		\\ &\lesssim
		\int_{-2}^{2} \Big(\sum_{j=1}^n  \frac{|f(x_j)|}{\sqrt{K_{n}({x}_j,{x_j})}} \frac{|K_{n\varepsilon}({x}_j,{x})| K_{n}({x}_j,{x}_j)}{K_{n\varepsilon}(x,x)^2}\cdot \frac{|K_{n\varepsilon}({x}_j,{x})|}{\sqrt{K_{n}({x}_j,{x}_j)}}\Big)^2\, dx \notag
		\\ &\lesssim \int_{-2}^{2}\Big(\sum_{j=1}^n  \frac{|f(x_j)|^2}{K_{n}({x}_j,{x_j})} \frac{K_{n\varepsilon}({x}_j,{x})^2 K_{n}({x}_j,{x}_j)^2}{K_{n\varepsilon}(x,x)^4}\Big) \Big( \sum_{j=1}^n  \frac{K_{n\varepsilon}({x}_j,{x})^2}{K_{n}({x}_j,{x}_j)} \Big)\, dx \notag
	\end{align*}
	Now applying Lemmas \ref{lemmabessel} and \ref{remsep}, and Lemma \ref{lemtech} for $n\varepsilon$ we conclude
	\begin{align*}
		\int_\R |f|^2 & \lesssim \int_{-2}^{2}\sum_{j=1}^n  \frac{|f(x_j)|^2}{K_{n}({x}_j,{x_j})} \frac{K_{n\varepsilon}({x}_j,{x})^2 K_{n}({x}_j,{x}_j)^2}{K_{n\varepsilon}(x,x)^3}\,  dx \notag
		\\ &\lesssim \frac{1}{\varepsilon^2} \sum_{j=1}^n  \frac{|f(x_j)|^2}{K_{n}({x}_j,{x_j})} \int_{-2}^{2} \frac{K_{n\varepsilon}({x}_j,{x})^2 K_{n\varepsilon}({x}_j,{x}_j)^2}{K_{n\varepsilon}(x,x)^3}\,  dx \notag
		\\ &\lesssim \frac{1}{\varepsilon^2} \sum_{j=1}^n  \frac{|f(x_j)|^2}{K_{n}({x}_j,{x_j})}. \qedhere
	\end{align*}
\end{proof}


\section{Discrepancy}\label{secdis}

In this section we prove Theorem~\ref{propdisc}.

\subsection{Toeplitz operators and rescaling}
For $n\in \N$ and an interval $\Omega\subseteq \R$, we let $T_{n,\Omega}: L^2(\R) \to L^2(\R)$ be the \emph{Toeplitz operator},
\[T_{n,\Omega}f= P_{\ww_n}(1_{\Omega} P_{\ww_n} f),\]
where $P_{\ww_n}$ is the orthogonal projection of $L^2(\R)$ onto $\ww_n$. Then $T_{n,\Omega}$ is a positive contraction
and for $f\in \ww_n$ we have
\[T_{n,\Omega} f(x)= P_{\ww_n}(1_{\Omega} f)(x)=\int_{\Omega}  K_n(x,y) f(y) dy.\]
Notice that $1_{\Omega} P_{\ww_n} 1_{\Omega}$ has the same spectrum as $T_{n,\Omega}$. We will consider $n$-dependent intervals $\Omega_n$ and
denote $T_{\rho n,\Omega_n}$ by $T_{\rho n}$ to alleviate notation.

For an interval $\Omega\subseteq \R$, consider the corresponding Toeplitz operators associated with the sine and the Airy kernels $\sker,\aker: L^2(\R) \to L^2(\R)$, given by the following integral kernels:
\begin{align}\label{eq_fa2}
	S(x,y)&=1_\Omega(x) \frac{\sin(\pi(x-y))}{\pi(x-y)} 1_\Omega(y),\qquad
	\text{and}\qquad
	\\ A(x,y)&=1_\Omega(x) \frac{\ai(x)\ai'(y)-\ai(y)\ai'(x)}{x-y} 1_\Omega(y), \notag
\end{align}
where $\ai$ is the Airy function (see \eqref{eqair}).

\begin{lemma}\label{coro_eigen2}
	Let $L>0$, and $\rho_* \in (1/2,3/2)$. For each $n \in \mathbb{N}$ let $\rho=\rho_n \in (1/2,3/2)$ with $\rho  n \in \N$ and $\rho \to \rho_*$. The following statements hold.
	\begin{enumerate}[label=(\roman*),ref=(\roman*)]
		\item (Bulk)  Fix $\delta>0$. For every $\varepsilon>0$, there exists an $n_0(\delta,L,\varepsilon)\in \N$ such that for $n\ge n_0$, $\eta>0$, $|x_0|\le 2-\delta$,  ${\Omega_n}=B(x_0,L/n)$ and $\Omega=B\big(0,\rho_* L\sigma(x_0)\big)$ 
		we have 	
		\begin{align}\label{eq_ggg2}
			\#\{k\in \N: \ \lambda_k(\sker)>\eta+\varepsilon\} &\le \#\{k\in \N: \ \lambda_k(T_{\rho n})>\eta\}
			\le \#\{k\in \N: \ \lambda_k(\sker)>\eta-\varepsilon\}.
		\end{align}
		\item (Boundary) For  $\eta>0$, $x_0=\pm 2$,  ${\Omega_n}=B\big(x_0,(L/n)^{2/3}\big)$ and $\Omega=B(0,(\rho_* L)^{2/3}),$
		we have
		\begin{multline*}
			\#\{k\in \N: \ \lambda_k(\aker)>\eta\}\le \liminf_{n\to\infty} 	\#\{k\in \N: \ \lambda_k(T_{\rho n})>\eta\} 
			\\ \le \limsup_{n\to\infty} 	\#\{k\in \N: \ \lambda_k(T_{\rho n})>\eta\}\le  \#\{k\in \N: \ \lambda_k(\aker)\ge\eta\}.
		\end{multline*}
	\end{enumerate}
\end{lemma}

\begin{proof}
	First notice that $T_{\rho n}$ has the same spectrum as $1_{\Omega_n} P_{\ww_{\rho n}} 1_{\Omega_n}$. We work with a rescaled version of this operator given by the kernel \begin{align}\label{eq_fbb2}
		1_{r({\Omega_n}-x_0)}(x)\frac{K_{\rho n}(x_0+x/r, x_0+y/r)}{r} 1_{r({\Omega_n}-x_0)}(y),
	\end{align}
	where $r=K_{\rho n}(x_0,x_0)$. By \cite[Theorem~1.2]{LeLu5} (bulk convergence of rescaled reproducing kernels) we get
	\begin{align*}
		\frac{K_{\rho n}(x_0+x/r, x_0+y/r)}{r} \longrightarrow \frac{\sin(\pi(x-y))}{\pi(x-y)},
		\qquad \mbox{as }n \to \infty,
	\end{align*}
	uniformly for $x,y \in r(\Omega_n-x_0)$, provided that $|x_0|\le 2-\delta$.
	Together with Lemma~\ref{remsep}, this readily implies that \eqref{eq_fbb2} converges in $L^2(\mathbb{R}\times\mathbb{R},dxdy)$ to the integral kernel in \eqref{eq_fa2}. This implies operator convergence, and in turn, convergence of the eigenvalues by a straightforward application of the Courant-Fisher formula. The boundary case follows analogously using \cite{Fo2} or \cite[Theorem 1.2]{LeLu6} and the rescaling $r=(\rho n)^{2/3}$.
\end{proof}

We first inspect the spectrum of Toeplitz operators related to the sine kernel. We will use the following refined asymptotics from~\cite{KaRoDa}.
\begin{lemma}\label{lemma_e1}
	There is a universal constant $C>0$ such that
	the following holds for $L>2$, $\Omega=B(0,L)$ and $\eta\in(0,1)$,
	\begin{align}\label{eqsin}
		\big|\#\{k\in \N: \ \lambda_k(\sker)>\eta\}-2L\big|  \le C \log(\max\{\eta^{-1},(1-\eta)^{-1}\}) \log L.
	\end{align}
\end{lemma}
Lemma \ref{lemma_e1} is proved in ~\cite[Theorem~3 and Section~3.2]{KaRoDa}, where in addition an explicit value for $C$ is given (see also \cite{israel15,Os,BoJaKa,KaZhWaRoDa}).
In contrast to more qualitative classic asymptotics going back to Landau and Widom \cite{MR593228, MR169015}, the merit of Lemma \ref{lemma_e1} is that the bound \eqref{eqsin} holds simultaneously for all $L$ and $\eta$. This fact will be crucial for the proof of Theorem~\ref{propdisc}.

The spectrum of the Airy integral operator seems to be less well understood than that of the sine kernel (see~\cite{KaAbMe,ShSe} for an analysis of numerical methods). We settle for the following estimate and point out that any improvement on the dependence of $\eta$ would directly translate to an improvement of the discrepancy on the boundary in Theorem~\ref{propdisc}.

\begin{lemma}\label{lemtrace2} For $\Omega=B(0, L^{\frac 23})$ with $L>2$ and $\eta\in (0,1)$ we have
	\[\big|\#\{k\in \N: \ \lambda_k(\aker)>\eta\}-\tfrac{2}{3\pi} L\big|=\max\{\eta^{-1},(1-\eta)^{-1}\} \oo(\log L).\]
\end{lemma}
\begin{proof}
	We first note that that for $\eta\in(0,1)$,
	\begin{align*}
		\big|\#\{k: \lambda_k(\aker)>\eta\}-\tr(\aker)\big|
		&=\big|\tr\big[(1_{(\eta,1)}-\id)(\aker)\big]\big|
		\le\tr\big[|1_{(\eta,1)}-\id|(\aker)\big] \notag
		\\ &\le \max\{\eta^{-1},(1-\eta)^{-1}\} \tr(\aker-\aker^2),
	\end{align*}
	since $|1_{(\eta,1)}(x)-x|\le \max\{\eta^{-1},(1-\eta)^{-1}\} (x-x^2)$. Therefore, it suffices to check that
	\[ \tr(\aker)=\frac{2}{3\pi} L+\oo(1), \quad \text{and} \quad \tr(\aker-\aker^2)=\oo(\log L).\]
	This statements are well-known because they describe the statistical fluctuations of
	the \emph{Airy random point field}, that is, the determinantal point process whose correlation kernel is the Airy kernel (see \cite{So}). Indeed, in terms of the Airy point field,
	\[ \tr(\aker)=\E \#[B(0, L^{\frac 23})], \quad \text{and} \quad \tr(\aker-\aker^2)=\V \#[B(0, L^{\frac 23})].\]
	In \cite{So} it is shown that
	\begin{align*}
		&\bullet \ \E \big[\#[(a,\infty)]^k\big] <\infty,  \text{ for every } a\in\R \text{ and } k\in\N;
		\\ &\bullet \  \E \#[(-L^{\frac 23},\infty)] =\tfrac {2 }{3\pi}L+\oo(1);
		\\ & \bullet  \ \V \#[(-L^{\frac 23},\infty)] =\oo(\log L).
	\end{align*}
	So we get
	\begin{align*}
		\big|\E \#[B(0,L^{\frac 23})]-\tfrac {2 }{3\pi}L\big| \le \E \#[ L^{\frac 23},\infty)] + \oo(1) \le \E \#[(0,\infty)] + \oo(1) = \oo(1),
	\end{align*}
	and,
	\begin{align*}
	\V \#[B(0,L^{\frac 23})] &\le 2\V \#[(-L^{\frac 23},\infty)]+2 \V \#[(L^{\frac 23},\infty)]
	\\ &
	\le \oo(\log L) + 2 \E\big[\#[(0,\infty)]^2\big]= \oo(\log L). \qedhere
\end{align*}
\end{proof}

\subsection{Proof of Theorem~\ref{propdisc}}
The proof is inspired by \cite[Proposition~1.3]{AmRom}.
Fix a bandwidth margin $\gamma>0$ and choose $\tau>0$ so that the assertions of Theorem~\ref{theosam} hold with probability $\mathbf{P}_n^{\beta_n}$ $\ge 1- C(n^{-\tau}+ e^{- \beta_n n /C})$. Let $\Omega_n,\Omega$ be as in Lemma~\ref{coro_eigen2} (in either the bulk or the boundary case).
Recall the disjoint intervals $I_j$ from Lemma~\ref{lemmabessel} and the corresponding constant $s$, and define
\begin{align*}
	J_n^+&=\{ j\in\{1,\ldots, n\}:\ I_j\cap\Omega_n \not= \varnothing\},
	\\ J_n^-&=\{j\in\{1,\ldots, n\}:\ I_j\subseteq \Omega_n \},
	\\  \ N_n^\pm&=\#J_n^\pm, \text{ and }N_n=\#[\Omega_n].
\end{align*}
Note that
\[N_n^-\le N_n \le N_n^+ \le N_n^- + 2.\]

\textit{Step 1. (Lower bound).} Choose $\rho_*=1-\gamma$ with
\begin{align}\label{eq_m0}
	0<\gamma<1/2,
\end{align}
and select $\rho=\rho_n=\rho_*+o(1)$ such that $\rho n \in \mathbb{N}$. Let us restrict the operator $T_{\rho n}$ to $\ww_{\rho n}$; this does not change its spectrum, except for the zero eigenvalue. Let $(\phi_j)_{j=1}^{\rho n}$ be an orthonormal basis of eigenvectors of $T_{\rho n}$ with eigenvalues
\[1\ge \lambda_1 \ge \ldots \ge \lambda_{\rho n}>0.\]
Here we assume that $\rho n$ is an integer as before.
Write
\[F=\spa\{\phi_1,\ldots,\phi_{N_n^+ +1}\}.\]
There is an element $f\in F$ such that $\|f\|=1$ and $f(x_j)=0$ for every $j\in J_n^+$.
By Theorem~\ref{theosam} and Lemma~\ref{lemmabessel}, for any $\alpha\in (0,1)$ we get
\begin{align*}
	1&=\|f\|^2 \le \frac{A}{ \gamma^{2}} \sum_{j\notin J_n^+} \frac{|f(x_j)|^2}{K_n(x_j,x_j)} \le \frac{A}{s\alpha \gamma^{2}} \sum_{j\notin J_n^+}\int_{I_j}|f|^2 +\frac{A  \tilde{C} s\alpha}{\gamma^{2}} \|f\|_2
	\\ &\le \frac{A}{s\alpha \gamma^{2}} \int_{\Omega^c_n}|f|^2 +\frac{A  \tilde{C} s\alpha}{\gamma^{2}}.
\end{align*}
Assume for the moment that
\begin{align}\label{eq_m1}
	\frac{A\tilde{C}s}{\gamma^{2}}>\frac12,
\end{align}
and choose $\alpha={\gamma^2}/{2A\tilde{C}s}$. Then
\begin{align*}
	1\le  \frac{C}{\gamma^{4}} \int_{\Omega^c_n}|f|^2,
\end{align*}
with $C=C(c_0,\tau)$, so that
\begin{align*}
	\lambda_{N_n^+ +1}\le \langle T_{\rho n}f,f \rangle =\int_{\Omega_n}|f|^2 =1-\int_{\Omega^c_n}|f|^2\le 1- \frac{\gamma^{4}}{C}.
\end{align*}
We may assume that $C>2$ and take $\eta=1-\gamma^{4}/C>1/2$. By Lemma \ref{coro_eigen2},
\begin{align*}
	\liminf_{n\to\infty} N_n^+ \ge \liminf_{n\to\infty} \#\{j: \lambda_j>\eta\}\ge
	\begin{cases}
		\#\{j: \lambda_j(\sker)>\eta\}
		& \text{(Bulk)}
		\\ \#\{j: \lambda_j(\aker)>\eta\}
		& \text{(Boundary)}
	\end{cases}.
\end{align*}
By Lemmas \ref{lemma_e1} and \ref{lemtrace2},
\begin{align}
	\label{eqlimi}
	\liminf_{n\to\infty} N_n \ge
	\begin{cases}
		\frac{\sqrt{4-x_0^2}}{\pi}(1-\gamma)L + \log (\gamma^{-1}) \oo( \log L)
		& \text{(Bulk)}
		\\ \frac{2}{3\pi}(1-\gamma) L + \gamma^{-4}\oo(\log L)
		& \text{(Boundary)}
	\end{cases}
\end{align}
where the implied constants do not depend on $\gamma$ or $x_0$.

\smallskip

\textit{Step 2. (Upper bound).} Set $\rho_*=1+\gamma$
with \eqref{eq_m0} and select $\rho=\rho_n = \rho_* + o(1)$ such that $\rho n \in \mathbb{N}$. Consider again $T_{\rho n}$ as an operator on $\ww_{\rho n}$. Let
\begin{alignat*}{2}
	&K_{x_j}&&=K_{\rho n}(\,\cdot\,,x_j)\in \ww_{\rho n},
	\\ &V&&=\spa\{K_{x_j}\}_{j=1}^n\subseteq \ww_{\rho n},
	\\ &W&&=V \ominus \spa\{K_{x_j}\}_{j\notin J_n^-} \subseteq \ww_{\rho n},
\end{alignat*}
and note that $\dim W=N_n^-$.

Now fix $f\in W \setminus \{0\}$. By Theorem~\ref{theosam} there exists $g\in\ww_{\rho n}$ such that $g(x_j)=f(x_j)$ for every $1\le j\le n$ and
\[\|g\|^2 \le \frac{A}{ \gamma^2} \sum_{j=1}^n \frac{|f(x_j)|^2}{K_n(x_j,x_j)}=\frac{A}{ \gamma^2} \sum_{j\in J_n^-} \frac{|f(x_j)|^2}{K_n(x_j,x_j)},\]
where we used that $f(x_j)=\langle f, K_{x_j}\rangle=0$ for $j\notin  J_n^-$. Letting $P_V:\ww_{\rho n}\to V$ be the orthogonal projection, we note that
\[P_V g (x_j)=\langle P_V g, K_{x_j}\rangle=\langle g, K_{x_j}\rangle=g (x_j)=f (x_j), \quad \text{for every }1\le j\le n.\]
Since evaluation at $(x_j)_{j=1}^n$ is an isomorphism between $V$ and $\C^n$, we have that $P_V g=f$ and
\[\|f\|^2\le \|g\|^2\le \frac{A}{ \gamma^2} \sum_{j\in J_n^-} \frac{|f(x_j)|^2}{K_n(x_j,x_j)}.\]
Combining this with Lemma~\ref{lemmabessel},
\[\|f\|^2\le \frac{A}{s \alpha \gamma^2} \int_{\Omega_n}|f|^2 +\frac{A  \tilde{C} s\alpha}{\gamma^{2}} \|f\|^2.\]
As before, assuming \eqref{eq_m1} and choosing $\alpha=\gamma^2/(2A\tilde{C}s)$ yields
\[\|f\|^2 < \frac{C}{\gamma^4}\int_{\Omega_n}|f|^2, \]
for a constant $C=C(c_0,\tau)$.
Since $\dim W=N_n^-$, by the Courant-Fisher characterization of eigenvalues of self-adjoint operators,
\[\lambda_{N_n^-}\ge \min_{f\in W \setminus \{0\}} \frac{\langle T_{\rho n}f,f\rangle}{\|f\|^2}= \min_{f\in W\setminus \{0\}} \frac{\int_{\Omega_n}|f|^2}{\|f\|^2}> \frac{\gamma^4}{C}.\]
By Lemma~\ref{coro_eigen2},
\begin{align*}
	\limsup_{n\to\infty} N_n^- \le \limsup_{n\to\infty} \#\{j: \lambda_j>\gamma^4/C\}\le
	\begin{cases}
		\#\{j: \lambda_j(\sker)\ge\gamma^4/C\}
		& \text{(Bulk)}
		\\ \#\{j: \lambda_j(\aker)\ge\gamma^4/C\}
		& \text{(Boundary)}
	\end{cases}.
\end{align*}
Assume also $C>1$ so that $\gamma^4/C<1/2$. Again by Lemmas \ref{lemma_e1} and \ref{lemtrace2},
\begin{align}
	\label{eqlims}
	\limsup_{n\to\infty} N_n \le
	\begin{cases}
		\frac{\sqrt{4-x_0^2}}{\pi}(1+\gamma)L + \log (\gamma^{-1})  \oo(\log L)
		& \text{(Bulk)}
		\\ \frac{2}{3\pi}(1+\gamma) L + \gamma^{-4}\oo(\log L)
		& \text{(Boundary)}
	\end{cases}.
\end{align}
where the implied constants do not depend on $\gamma$ or $x_0$.

\smallskip

\textit{Step 3. (Conclusions).} From \eqref{eqlimi} and \eqref{eqlims},
\begin{align*}
	\bullet&& \limsup_{n\to\infty} \Big| N_n - \tfrac{\sqrt{4-x_0^2}}{\pi}L \Big|&\le
	\gamma\oo(L) + \log (\gamma^{-1}) \oo(\log L)
	& &\text{(Bulk)};
	\\\bullet&& \limsup_{n\to\infty} \Big| N_n - \tfrac{2}{3\pi} L \Big| &\le
	\gamma \oo(L) + \gamma^{-4}\oo(\log L)
	& &\text{(Boundary)}.
\end{align*}
Moreover, by Lemma~\ref{coro_eigen2}, in the bulk case, the bound is uniform in the following sense: for $\delta>0$ there exists $n_0(\delta,L,\gamma,\tau,c_0)\in \N$ such that for every $n\ge n_0$ and $|x_0|\le 2-\delta$,
\[\Big| N_n - \tfrac{\sqrt{4-x_0^2}}{\pi}L \Big|\le
\gamma\oo(L) + \log (\gamma^{-1}) \oo(\log L).\]

The desired conclusions \eqref{eq_ppp1} and \eqref{eq_ppp2} follow by taking $\gamma=L^{-1}$ in the bulk case and $\gamma=(\log L /L)^{1/5}$ in the boundary case. These choices are compatible with \eqref{eq_m0} and \eqref{eq_m1} as soon as $L \geq L_0$ with $L_0=L_0(c,\tau)$ a certain constant. For the range $2<L<L_0$ we use a crude bound based on the separation condition \ref{eqsepar} of Theorem~\ref{theosam}. \qed

\section{Proof of Lemma~\ref{lemtech}} \label{seclem}
In this section we prove Lemma~\ref{lemtech}. First notice that by the reproducing property and Cauchy–Schwarz inequality:
\begin{align}\label{eqcs}
	K_n(x,y)^2\le
	K_n(x,x) K_n(y,y), \quad x,y \in \mathbb{R}.
\end{align}
This can also be seen by noting that the two-point intensity of the GUE is non-negative and equals the determinant $K_n(x,x) K_n(y,y)-K_n(x,y)^2$ (see \cite[Proposition 5.1.2]{Fo}).

We also need a more precise control of $K_n(x,y)$.
An orthonormal basis of $\mathcal{W}_n$ is given by the rescaled Hermite functions $\{p_k\}_{k=0}^{n-1}$, where
\begin{align}\label{eqherm}
	p_k(x) =  n^{1/4} h_k\!\left(\sqrt{n}x\right)e^{-nx^2/4},  \qquad x \in \mathbb{R},\ k \geq 0,
\end{align}
and
\begin{align*}
	h_k(x) = \big(\sqrt{2\pi} k!\big)^{-\tfrac{1}{2}} (-1)^k e^{x^2/2} \frac{d^k}{d^kx} e^{-x^2/2}, \qquad x \in \mathbb{R},\ k \geq 0.
\end{align*}

From the \emph{Christoffel-Darboux formula} (see \cite[Proposition~5.1.3]{Fo}) we have
\begin{align}
	\label{eqCD}
	K_n(x,y)=\sum_{k=0}^{n-1}p_k(x)p_k(y)=\frac{p_n(x)p_{n-1}(y)-p_n(y)p_{n-1}(x)}{x-y}.
\end{align}
So, estimates on Hermite polynomials such as \cite[Theorem~12.1]{LeLu2}, \cite[Theorem~8.22.9]{Sz} and \cite[Theorem~2.2]{DKMLV} can be used to obtain off-diagonal estimates for the reproducing kernel $K_n(x,y)$. For the first two we do this by hand, and for the latter we elaborate on what was already done in \cite{Gu}. In general, off-diagonal estimates have been much less studied than the microscopic limits (the sine kernel and the Airy kernel), and while the following bounds
are most probably known among experts, we are not aware of more explicit references.

\begin{lemma}\label{lpn}
	The following estimates hold.
	\begin{enumerate}
		\item \label{it11} For every $n\in \N$ and $|x|\le 2-n^{-2/3}$ we have
		\begin{align}
			\label{eqkernel11}
			p_n(x)&= \oo\Big( \frac{1}{(2-|x|)^{1/4}} \Big),
			\intertext{and}
			\label{eqkernel12}
			p_n(x)-p_{n-1}(x)&= \oo\big( (2-|x|)^{1/4} \big).
		\end{align}
		\item \label{it21} Fix $M\ge 1$. For every $n\in \N$ and every $x\in \R$ such that $|2- |x| | \le M n^{-2/3}$ we have
		\begin{align}
			\label{eqkernel21}
			p_n(x)&= \oo\big( n^{1/6} \big),
			\intertext{and}
			\label{eqkernel22}
			p_n(x)-p_{n-1}(x)&= \oo\big( n^{-1/6} \big),
		\end{align}
		where the implicit constants depend only on $M$.
	\end{enumerate}
\end{lemma}

\begin{proof}
	From \cite[Theorem~12.1]{LeLu2}, for every $|t|< 2\sqrt{n}$ we have
	\begin{align*}
		|h_n(t)| e^{-t^2/4}
		\lesssim \frac{1}{(4n-t^2)^{1/4}}\lesssim \frac{1}{n^{1/4}(2-|t|/\sqrt{n})^{1/4}}.
	\end{align*}
	For $|x|< 2$, taking $t=\sqrt{n}x$ we get
	\begin{align}\label{epn}
		|p_n(x)|=n^{1/4} |h_n\!\left(\sqrt{n}x\right)|e^{-nx^2/4} \lesssim  \frac{1}{(2-|x|)^{1/4}},
	\end{align}
	which proves \eqref{eqkernel11}.
	
	The proof of \eqref{eqkernel12} is essentially implicit in the proof of \cite[Lemma~2.3]{Gu}. We provide a few details for completeness. From the asymptotics of \cite[Theorem~2.2]{DKMLV} (see also \cite{Gu}) one gets that for a sufficiently small $\delta>0$, if $2-\delta \le x \le 2-n^{-2/3}$, then
	\begin{align*}
		\bullet && p_n(x)&=
		\begin{multlined}[t]
			\Big(\frac{2+x}{2-x}\Big)^{1/4}(3nF(x))^{1/6}
			\ai(-(3nF(x))^{2/3})(1+\oo(n^{-1}))
			\\ +\Big(\frac{2-x}{2+x}\Big)^{1/4}(3nF(x))^{-1/6}
			\ai'(-(3nF(x))^{2/3})(1+\oo(n^{-1}));
		\end{multlined}
		\\
		\bullet && p_{n-1}(x)&=\begin{multlined}[t] \Big(\frac{n}{n'}\Big)^{1/4}
			\Big[\Big(\frac{2+x_n}{2-x_n}\Big)^{1/4}(3n'F(x_n))^{1/6}
			\ai(-(3n'F(x_n))^{2/3})(1+\oo(n^{-1}))
			\\ +\Big(\frac{2-x_n}{2+x_n}\Big)^{1/4}(3n'F(x_n))^{-1/6}
			\ai'(-(3n'F(x_n))^{2/3})(1+\oo(n^{-1}))\Big];
		\end{multlined}
	\end{align*}
	where $n'=n-1$, $x_n=x\sqrt{n/n'}$, $\ai$ is the Airy function given by 
	\begin{align}\label{eqair}
		\ai(x)=\frac{1}{\pi} \int_0^\infty \cos(t^3/3+xt)dt,
	\end{align}
	and
	\[F(x)=\frac{1}{4}\int_x^2\sqrt{4-t^2}dt.\]
	Fix $2-\delta \le x \le 2-n^{-2/3}$, so that these estimates hold. Since  $x_n=x+\oo(n^{-1})$, a Taylor expansion around $x$ shows that
	\begin{align*}
		\bullet && \Big(\frac{2+x_n}{2-x_n}\Big)^{1/4}(3n'F(x_n))^{1/6}=&
		\Big(\frac{2+x}{2-x}\Big)^{1/4}(3nF(x))^{1/6} +\oo(n^{-5/6});
		\\ \bullet && \Big(\frac{2-x_n}{2+x_n}\Big)^{1/4}(3n'F(x_n))^{-1/6}=&
		\Big(\frac{2-x}{2+x}\Big)^{1/4}(3nF(x))^{-1/6} +\oo(n^{-5/6}).
	\end{align*}
	
	Regarding the Airy function, from \cite[Chapter~11]{Ol}, for $t\gtrsim 1$ we have
	\begin{align}\label{eai}
		|\ai(-t)|\lesssim t^{-1/4}, \quad
		|\ai'(-t)|\lesssim t^{1/4} \quad \text{and} \quad
		|\ai''(-t)|\lesssim t^{3/4}.
	\end{align}
	Notice that
	\begin{align*}
		F(x)\simeq (2-x)\sqrt{4-x^2}\simeq (2-x)^{3/2}.
	\end{align*}
	In particular, $(3nF(x))^{2/3}\gtrsim 1$ and the estimates \eqref{eai} hold near this value. Furthermore, (after a suitable rescaling) in \cite[Lemma~2.3]{Gu} it is shown that
	\begin{align*}
		(3n'F(x_n))^{2/3}= (3nF(x))^{2/3} + \oo(n^{-1/3}).
	\end{align*}
	Expanding $\ai$ and $\ai'$ we get
	\begin{align*}
		\bullet && \ai(-(3n'F(x_n))^{2/3})&
		=\ai(-(3nF(x))^{2/3}) +\oo(n^{-1/6}(2-x)^{1/4});
		\\ \bullet && \ai'(-(3n'F(x_n))^{2/3})&=
		\ai'(-(3nF(x))^{2/3}) +\oo(n^{1/6}(2-x)^{3/4}).
	\end{align*}
	Joining everything,
	\begin{align*}
		p_{n-1}(x)&= \Big(\frac{n}{n'}\Big)^{1/4} (p_n(x)+\oo((2-x)^{1/4}))=p_n(x)+\oo((2-x)^{1/4})+\oo(n^{-3/4})
		\\ &=p_n(x)+\oo((2-x)^{1/4}).
	\end{align*}
	This proves \eqref{eqkernel12} for $2-\delta \le x \le 2-n^{-2/3}$. The case $-2+n^{-2/3} \le x \le -2+\delta$ is completely analogous. If $|x| \le 2-\delta$, again by \cite[Theorem~12.1]{LeLu2} for $|t|<2\sqrt{n-1}$ we get
	\begin{align*}
		|h_{n-1}(t)| e^{-t^2/4}
		\lesssim \frac{1}{(4(n-1)-t^2)^{1/4}}\lesssim \frac{1}{n^{1/4}(2-|t|/\sqrt{n-1})^{1/4}}.
	\end{align*}
	Since $|x|\le 2-\delta$, taking $t=\sqrt{n}x$ we get that $|t|\le (2-\delta/2)\sqrt{n-1}$ if $n$ is sufficiently large. Therefore,
	\begin{align*}
		|p_{n-1}(x)|&=n^{1/4} |h_{n-1}\!\left(\sqrt{n}x\right)|e^{-nx^2/4} \lesssim  \frac{1}{n^{1/4}(2-|t|/\sqrt{n-1})^{1/4}} =\oo(1)
		\\&=\oo((2-|x|)^{1/4}),
	\end{align*}
	where we leave the dependence on $\delta$ implicit since this parameter is fixed throughout the proof.
	This together with \eqref{epn} gives
	\[p_n(x)-p_{n-1}(x)=\oo(1)=\oo((2-|x|)^{1/4}).\]
	Finally, for small $n$ and $|x| \le 2-\delta$ it is clear that
	\[p_n(x)-p_{n-1}(x)=\oo((2-|x|)^{1/4}),\]
	which completes the proof of \eqref{eqkernel12}.
	
	Now we proceed to show \eqref{eqkernel21}. Rescaling \cite[Theorem~8.22.9]{Sz}, since $x=2+\oo(n^{-2/3})$ for $k=n,n-1$ we have that
	\begin{align}\label{epk}
		p_k(x)=n^{1/4} k^{-1/12} (\ai(-s_k)+\oo(n^{-2/3})),
	\end{align}
	where
	\[s_k=k^{1/6}\Big(\sqrt{k+\frac{1}{2}}-\frac{\sqrt{n}x}{2}\Big).\]
	Note that
	\[s_k=k^{1/6}\Big(\sqrt{n}+\oo(n^{-1/2})-\sqrt{n}+\oo(n^{-1/6})\Big)=k^{1/6}\oo(n^{-1/6})=\oo(1).\]
	Since $\ai$ is continuous and $s_n$ is bounded by a constant depending only on $M$, \eqref{eqkernel21} follows from \eqref{epk}.
	
	Regarding \eqref{eqkernel22}, we have
	\begin{align*}
		s_{n-1}&=(n-1)^{1/6}\Big(\sqrt{n-\frac{1}{2}}-\frac{\sqrt{n}x}{2}\Big)=n^{1/6}\Big(1-\frac{1}{n}\Big)^{1/6}\Big(\sqrt{n-\frac{1}{2}}-\frac{\sqrt{n}x}{2}\Big)
		\\&=(1+\oo(n^{-1}))s_n  (1+\oo(n^{-1}))n^{1/6}\Big(\sqrt{n-\frac{1}{2}}-\sqrt{n+\frac{1}{2}}\Big)
		\\&=s_n+\oo(n^{-1})+ \oo(n^{-1/3})=s_n+ \oo(n^{-1/3}).
	\end{align*}
	So, expanding $\ai(s_{n-1})$,
	\[\ai(-s_{n-1})=\ai(-s_{n})+\oo(n^{-1/3}).\]
	Plugging this into \eqref{epk},
	\begin{align*}
		p_{n-1}(x)&=n^{1/4} (n-1)^{-1/12} (\ai(-s_{n-1})+\oo(n^{-2/3}))
		\\ &=n^{1/4} n^{-1/12}(1+\oo(n^{-1})) (\ai(-s_{n})+\oo(n^{-1/3}))
		\\ &=(1+\oo(n^{-1}))p_n(x)+\oo(n^{-1/6})=p_n(x)+\oo(n^{-1/6}),
	\end{align*}
	where in the last step we used \eqref{eqkernel21}.
\end{proof}

\begin{lemma}[Off-diagonal estimates for the reproducing kernel]
	\label{leker}
	The following estimates hold.
	\begin{enumerate}
		\item \label{it1} For every $n\in \N$ and $|y|\le |x|\le 2-n^{-2/3}$ we have
		\begin{align}
			\label{eqkernel1}
			K_n(x,y)= \oo\Big( \frac{(2-|y|)^{1/4}}{(2-|x|)^{1/4}|x-y|} \Big).
		\end{align}
		\item \label{it2}  Fix $M\ge 1$.  For every $n\in \N$ and every $|y|\le 2- n^{-2/3}\le |x| \le 2 + M n^{-2/3}$ we have
		\begin{align}
			\label{eqkernel2}
			K_n(x,y)= \oo\Big( \frac{n^{1/6} (2-|y|)^{1/4} }{|x-y|} \Big),
		\end{align}
		where the implicit constant depends only on $M$.
	\end{enumerate}
\end{lemma}

\begin{proof}
	From the Christoffel-Darboux formula \eqref{eqCD},
	\[|K_n(x,y)|\le \frac{|p_n(x)||p_{n-1}(y)-p_n(y)|}{|x-y|}+ \frac{|p_n(y)||p_n(x)-p_{n-1}(x)|}{|x-y|}.\]
	Applying Lemma \ref{lpn} to $x$ and $y$ according to the region where they lie gives the result. We only point out that for the second statement one has to use the bound
	\[n^{-1/3} \le \sqrt{2-|y|}. \qedhere\]
\end{proof}

\begin{remark} For $|x|,|y|\le 2-\delta$, \eqref{eqkernel1} reduces to $|K_n(x,y)|=O(|x-y|^{-1})$. This off-diagonal bulk estimate away from the boundary is known for a broad class of potentials, see for example Kuijlaars' lecture notes \cite[Eq. (5.44)]{Kuijlaars}.
\end{remark}

Now we are in position to prove Lemma~\ref{lemtech}.

\begin{proof}[Proof of Lemma~\ref{lemtech}]
	Let us start with the first assertion. If $|x|\le|y|$, by Lemma~\ref{lspinoff} we have that $K_n(y,y)\lesssim K_n(x,x)$ and thus
	\begin{align}
		\label{eqint3}
		\int_{\{x:|x|\le |y| \}} \frac{K_n(x,y)^2}{K_n(x,x)^3}\, dx\lesssim \frac{1}{K_n(y,y)^3}\int_\R K_n(x,y)^2\,  dx = \frac{1}{K_n(y,y)^2}.
	\end{align}
	So we only need to estimate
	\begin{align}
		\label{eqint33}
		\int_{\{x:|y|\le |x| \le 2\}}\frac{K_n(x,y)^2}{K_n(x,x)^3}\, dx.
	\end{align}
	
	Notice that if $ 2- 2 n^{-2/3} \le |y|\le |x| \leq 2$, then $K_n(x,x) \simeq K_n(y,y)\simeq n^{2/3}$
	and
	\begin{align}
		\label{eqint32}
		\int_{\{x:|y|\le |x| \le 2\}} \frac{K_n(x,y)^2}{K_n(x,x)^3}\, dx &\le
		\int_{\{x:|y|\le |x| \le 2\}} \frac{K_n(y,y)}{K_n(x,x)^2}\, dx
		\\ &\lesssim n^{-4/3} \simeq \frac{1}{K_n(y,y)^2}. \notag
	\end{align}
	
	It remains to study the integral in \eqref{eqint33} for $|y|\le 2-2n^{-2/3}$. So we fix an arbitrary point $y$ satisfying $|y|\le 2-2n^{-2/3}$. We divide our analysis into three parts,
	according to the distance to the point $y$ and the position of $x$. For notational simplicity, the regions considered are a covering rather than a partition of $\{x:|y|\le |x| \le 2\}$.
	
	\textbf{Case 1:} first we deal with the values of $x$ such that $|x-y|\le K_n(y,y)^{-1}$. In this case, by Lemma \ref{lspinoff}, \eqref{eqcs} and Remark~\ref{lemma_a} we have that $K_n(y,y)\simeq K_n(x,x)$ and $K_n(x,y)^2\le K_n(x,x)K_n(y,y)\simeq K_n(y,y)^2$. So we get
	\begin{align}
		\label{eqint4}
		\int_{\{x: |x-y|\le K_n(y,y)^{-1}\}} \frac{K_n(x,y)^2}{K_n(x,x)^3}\, dx\lesssim
		\frac{1}{K_n(y,y)^2}.
	\end{align}
	\textbf{Case 2:}
	for $|x|\le 2-n^{-2/3}$, using \eqref{eqkernel1} and Lemma \ref{lspinoff},
	
	\[\frac{K_n(x,y)^2}{K_n(x,x)^3} \lesssim  \frac{(2-|y|)^{1/2}}{n^3(2-|x|)^{2}(x-y)^2}.\]
	Therefore, denoting $A_1=\{x\in\R: |x-y|\geq K_n(y,y)^{-1}, |x|\le 2- n^{-2/3}\}$,
	\begin{align*}
		\int_{A_1} \frac{K_n(x,y)^2}{K_n(x,x)^3}\, dx \lesssim \frac{(2-|y|)^{1/2}}{n^3}
		\int_{A_1} \frac{1}{(x-y)^2(2-|x|)^{2}}\, dx.
	\end{align*}
	Notice that
	\[2-|y|=2-|y-x+x|\le 2-(|x|-|x-y|)=2-|x|+|x-y|.\]
	So,
	\begin{align}\label{eqk3}
		\int_{A_1} \frac{K_n(x,y)^2}{K_n(x,x)^3}\, dx
		&\lesssim  \frac{1}{  n^3(2-y)^{3/2}} \int_{A_1} \frac{(2-|y|)^2}{|x-y|^2(2-|x|)^{2}}\, dx
		\\ &\le   \frac{1}{  n^3(2-y)^{3/2}}\int_{A_1}\Big( \frac{1}{|x-y|} + \frac{1}{2-|x|}\Big)^2 dx \notag
		\\ &\lesssim   \frac{1}{  n^3(2-y)^{3/2}}\Big(\int_{A_1} \frac{1}{(x-y)^2}\, dx + \int_{A_1} \frac{1}{(2-|x|)^2}\,  dx \Big)\notag
		\\ & \lesssim \frac{1}{ n^3(2-y)^{3/2}} \Big( \int_{K_n(y,y)^{-1}}^\infty \frac{1}{t^2}\, dt+ \int_{n^{-2/3}}^{\infty } \frac{1}{t^2}\, dt\Big) \notag
		\\ &= \frac{1}{ n^3(2-y)^{3/2}}  \left(K_n(y,y) +n^{2/3}\right)
		\lesssim \frac{1}{ K_n(y,y)^{3}}  K_n(y,y)\notag \\
		&
		= \frac{1}{ K_n(y,y)^{2}}  . \notag
	\end{align}
	
	\textbf{Case 3:}
	denote $A_2=\{x\in\R: 2- n^{-2/3}\le |x|\le 2\}$. Since $|y|\le 2-2n^{-2/3}$; a straightforward computation shows that
	\[2-|y|\le 2(2-n^{-2/3}-|y|)\le 2(|x|-|y|)\le 2 |x-y|.\]
	For $x\in A_2$, using this together with \eqref{eqkernel2} and Lemma \ref{lspinoff} we get
	\[\frac{K_n(x,y)^2}{K_n(x,x)^3} \lesssim  \frac{n^{1/3}(2-|y|)^{1/2}}{n^{2}(x-y)^2}\lesssim \frac{n^{1/3}}{n^{2}(2-|y|)^{3/2}}\lesssim \frac{n^{1/3}}{K_n(y,y)^{2}(2-|y|)^{1/2}}\lesssim \frac{n^{2/3}}{K_n(y,y)^{2}}.\]
	So,
	\begin{align*}
		\int_{A_2}  \frac{K_n(x,y)^2}{K_n(x,x)^3}\, dx
		\lesssim \frac{n^{2/3}}{  K_n(y,y)^{2}} |A_2|
		\lesssim \frac{1}{ K_n(y,y)^{2}}  .
	\end{align*}
	
	This together with \eqref{eqint3}, \eqref{eqint32}, \eqref{eqint4} and \eqref{eqk3}, prove \eqref{eqint}.

	The proof of \eqref{eqint2} is similar. We begin by fixing $y$ with $|y|\le 2-Mn^{-2/3}$. If $|x|\ge|y|$, then by Lemma \ref{lspinoff} and \eqref{eqcs} we have that $K_n(x,x)\lesssim K_n(y,y)$ and $K_n(x,y)^2 \le K_n(x,x) K_n(y,y)\lesssim K_n(y,y)^2$. Thus,
	\begin{align}
		\label{eqint30}
		\int_{\{x:|y|\le |x| \le 2\}} K_n(x,y)^6 K_n(x,x)\, dx\lesssim K_n(y,y)^5 \int_\R K_n(x,y)^2\, dx = K_n(y,y)^6.
	\end{align}
	So we only need to estimate
	\[ \int_{\{x:|x|\le |y| \}} K_n(x,y)^6 K_n(x,x)\, dx.\]
	
	As before, we divide our analysis according to the distance between $x$ and $y$, and the position of $y$.
	
	\textbf{Case 1:} consider first the values of $x$ such that $|x-y|\le K_n(y,y)^{-1}$. By Lemma~\ref{lspinoff}, \eqref{eqcs} and Remark \ref{lemma_a}, $K_n(y,y)\simeq K_n(x,x)$ and $K_n(x,y)^2\lesssim K_n(x,x)K_n(y,y)\simeq K_n(y,y)^2$. We get
	\begin{align}
		\label{eqint8}
		\int_{\{x: |x-y|\le K_n(y,y)^{-1}, \, |x|\le|y|\}}  & K_n(x,y)^6 K_n(x,x)\, dx\\
		&\lesssim
		K_n(y,y)^7 K_n(y,y)^{-1} = K_n(y,y)^{6}.\nonumber
	\end{align}
	\textbf{Case 2:}
	if $|y|\le 2-n^{-2/3}$, using \eqref{eqkernel1} and Lemma \ref{lspinoff},
	\[K_n(x,y)^6 K_n(x,x) \lesssim  \frac{n (2-|x|)^{2}}{(2-|y|)^{3/2}|x-y|^6}.\]
	Define $A_3:=\{x\in\R: |x-y|\geq K_n(y,y)^{-1}, |x|\le |y|\}$. Since by the triangle inequality $2-|x|\le 2-|y|+|x-y|$, we get
	\begin{align}
		\label{eqint9}
		\int_{A_3}  & K_n(x,y)^6  K_n(x,x)\, dx
		\lesssim \int_{A_3} \frac{n(2-|x|)^2}{ (2-|y|)^{3/2}|x-y|^6}\, dx
		\\ & \le \frac{n }{(2-|y|)^{3/2}}\int_{A_3} \frac{(x-y)^2 + 2 |x-y|(2-|y|)+ (2-|y|)^2}{|x-y|^6}\, dx \notag
		\\ & \lesssim n\Big(\frac{K_n(y,y)^3}{(2-|y|)^{3/2}}+ \frac{K_n(y,y)^4}{(2-|y|)^{1/2}} +K_n(y,y)^5(2-|y|)^{1/2} \Big) \notag
		\\ &\lesssim n^2 K_n(y,y)^3 + n^{4/3} K_n(y,y)^4 +  K_n(y,y)^6 \lesssim K_n(y,y)^6, \notag
	\end{align}
	where in the last step we used $n^{2/3}\lesssim K_n(y,y)$.
	
	\textbf{Case 3:} assume $|y|\ge 2-n^{-2/3}$. Notice that if also  $|x|\ge 2-n^{-2/3}$, then $K_n(x,x) \simeq K_n(y,y)\simeq n^{2/3}$
	and
	\begin{align}
		\label{eqint36}
		\int_{\{x:\ 2-  n^{-2/3} \le |x|\le |y|\}} & K_n(x,y)^6 K_n(x,x)\, dx \\
		&\le \int_{\{x:\ 2-  n^{-2/3} \le |x|\le |y|\}} K_n(y,y)^3 K_n(x,x)^4\, dx\notag
		\\ &\lesssim  K_n(y,y)^7 n^{-2/3}\simeq  K_n(y,y)^6. \notag
	\end{align}
	So it remains to deal with the region $A_4:=\{x\in\R: |x-y|\geq K_n(y,y)^{-1}, |x|\le 2-n^{-2/3}\}$.
	Using again that $2-|x|\le 2-|y|+|x-y|$ together with \eqref{eqkernel2}
	and Lemma~\ref{lspinoff},
	\begin{align*}
		\int_{A_4}  K_n(x,y)^6 & K_n(x,x)\, dx
		\lesssim \int_{A_4} \frac{n^2(2-|x|)^2}{ |x-y|^6}\, dx \notag
		\\ & = n^2\int_{A_4} \frac{(x-y)^2 + 2 |x-y|(2-|y|)+ (2-|y|)^2}{|x-y|^6}\, dx \notag
		\\ & \lesssim n^2\Big(K_n(y,y)^3+ K_n(y,y)^4(2-|y|) +K_n(y,y)^5(2-|y|)^{2} \Big) \notag
		\\ &\lesssim n^2 K_n(y,y)^3 + n^{4/3} K_n(y,y)^4 +  n^{2/3}K_n(y,y)^5 \lesssim K_n(y,y)^6,
	\end{align*}
	where in the last step we used $n^{2/3}\lesssim K_n(y,y)$. This together with \eqref{eqint30},
	\eqref{eqint8}, \eqref{eqint9} and \eqref{eqint36} prove \eqref{eqint2}.
\end{proof}

\bibliographystyle{abbrv}
\bibliography{biblio}

\begin{thebibliography}{10}

\bibitem{AGK22}
G.~Akemann, V.~Gorski, and M.~Kieburg.
\newblock Consecutive level spacings in the chiral {G}aussian unitary ensemble:
  from the hard and soft edge to the bulk.
\newblock {\em J. Phys. A}, 55(19):Paper No. 194002, 34, 2022.

\bibitem{AMP}
G.~Akemann, A.~Mielke, and P.~P\"{a}\ss~ler.
\newblock Spacing distribution in the two-dimensional {C}oulomb gas: surmise
  and symmetry classes of non-{H}ermitian random matrices at noninteger
  {$\beta$}.
\newblock {\em Phys. Rev. E}, 106(1):Paper No. 014146, 11, 2022.

\bibitem{Am2}
Y.~Ameur.
\newblock Repulsion in low temperature {$\beta$}-ensembles.
\newblock {\em Comm. Math. Phys.}, 359(3):1079--1089, 2018.

\bibitem{Am}
Y.~Ameur.
\newblock A localization theorem for the planar {C}oulomb gas in an external
  field.
\newblock {\em Electron. J. Probab.}, 26:Paper No. 46--21, 2021.

\bibitem{AB21}
Y.~Ameur and S.-S. Byun.
\newblock Almost-{H}ermitian random matrices and bandlimited point processes.
\newblock {\em Anal. Math. Phys.}, 13(3):Paper No. 52, 57, 2023.

\bibitem{AmOC}
Y.~Ameur and J.~Ortega-Cerd\`a.
\newblock Beurling-{L}andau densities of weighted {F}ekete sets and correlation
  kernel estimates.
\newblock {\em J. Funct. Anal.}, 263(7):1825--1861, 2012.

\bibitem{AmRom}
Y.~Ameur and J.~L. Romero.
\newblock The planar low temperature {C}oulomb gas: separation and
  equidistribution.
\newblock {\em Rev. Mat. Iberoam.}, 39(2):611--648, 2023.

\bibitem{AmTr}
Y.~Ameur and E.~Troedsson.
\newblock Remarks on the one-point density of hele-shaw $\beta$-ensembles.
\newblock {\em arXiv preprint:2402.13882}, 2024.

\bibitem{AS21}
S.~Armstrong and S.~Serfaty.
\newblock Local laws and rigidity for {C}oulomb gases at any temperature.
\newblock {\em Ann. Probab.}, 49(1):46--121, 2021.

\bibitem{BaSu}
Z.~Bao and Z.~Su.
\newblock Local semicircle law and gaussian fluctuation for hermite {$\beta$}
  ensemble.
\newblock {\em Sci. Sin. Math.}, 42(10):1017--1030, 2012.

\bibitem{BeLeSe}
F.~Bekerman, T.~Lebl\'{e}, and S.~Serfaty.
\newblock C{LT} for fluctuations of {$\beta$}-ensembles with general potential.
\newblock {\em Electron. J. Probab.}, 23:Paper no. 115, 31, 2018.

\bibitem{BeLo}
F.~Bekerman and A.~Lodhia.
\newblock Mesoscopic central limit theorem for general {$\beta$}-ensembles.
\newblock {\em Ann. Inst. Henri Poincar\'{e} Probab. Stat.}, 54(4):1917--1938,
  2018.

\bibitem{BoJaKa}
A.~Bonami, P.~Jaming, and A.~Karoui.
\newblock Non-asymptotic behavior of the spectrum of the sinc-kernel operator
  and related applications.
\newblock {\em J. Math. Phys.}, 62(3):Paper No. 033511, 20, 2021.

\bibitem{BHS}
S.~V. Borodachov, D.~P. Hardin, and E.~B. Saff.
\newblock {\em Discrete energy on rectifiable sets}.
\newblock Springer Monographs in Mathematics. Springer, New York, [2019]
  \copyright 2019.

\bibitem{BoErYa}
P.~Bourgade, L.~Erd\H{o}s, and H.-T. Yau.
\newblock Universality of general {$\beta$}-ensembles.
\newblock {\em Duke Math. J.}, 163(6):1127--1190, 2014.

\bibitem{BoErYa2}
P.~Bourgade, L.~Erd\"{o}s, and H.-T. Yau.
\newblock Edge universality of beta ensembles.
\newblock {\em Comm. Math. Phys.}, 332(1):261--353, 2014.

\bibitem{CSA}
G.~Cardoso, J.-M. St\'{e}phan, and A.~G. Abanov.
\newblock The boundary density profile of a {C}oulomb droplet. {F}reezing at
  the edge.
\newblock {\em J. Phys. A}, 54(1):Paper No. 015002, 24, 2021.

\bibitem{MR3831027}
T.~Carroll, J.~Marzo, X.~Massaneda, and J.~Ortega-Cerd\`a.
\newblock Equidistribution and {$\beta$}-ensembles.
\newblock {\em Ann. Fac. Sci. Toulouse Math. (6)}, 27(2):377--387, 2018.

\bibitem{DKMLV}
P.~Deift, T.~Kriecherbauer, K.~T.-R. McLaughlin, S.~Venakides, and X.~Zhou.
\newblock Strong asymptotics of orthogonal polynomials with respect to
  exponential weights.
\newblock {\em Comm. Pure Appl. Math.}, 52(12):1491--1552, 1999.

\bibitem{DeKrMLVeZh2}
P.~Deift, T.~Kriecherbauer, K.~T.-R. McLaughlin, S.~Venakides, and X.~Zhou.
\newblock Uniform asymptotics for polynomials orthogonal with respect to
  varying exponential weights and applications to universality questions in
  random matrix theory.
\newblock {\em Comm. Pure Appl. Math.}, 52(11):1335--1425, 1999.

\bibitem{De99}
P.~A. Deift.
\newblock {\em Orthogonal polynomials and random matrices: a
  {R}iemann-{H}ilbert approach}, volume~3 of {\em Courant Lecture Notes in
  Mathematics}.
\newblock New York University, Courant Institute of Mathematical Sciences, New
  York; American Mathematical Society, Providence, RI, 1999.

\bibitem{DH90}
B.~Dietz and F.~Haake.
\newblock Taylor and {P}ad\'{e} analysis of the level spacing distributions of
  random-matrix ensembles.
\newblock {\em Z. Phys. B}, 80(1):153--158, 1990.

\bibitem{DuEd}
I.~Dumitriu and A.~Edelman.
\newblock Matrix models for beta ensembles.
\newblock {\em J. Math. Phys.}, 43(11):5830--5847, 2002.

\bibitem{ErXu}
L.~Erd\H{o}s and Y.~Xu.
\newblock Small deviation estimates for the largest eigenvalue of {W}igner
  matrices.
\newblock {\em Bernoulli}, 29(2):1063--1079, 2023.

\bibitem{ErYa}
L.~Erd\H{o}s and H.-T. Yau.
\newblock Gap universality of generalized {W}igner and {$\beta$}-ensembles.
\newblock {\em J. Eur. Math. Soc. (JEMS)}, 17(8):1927--2036, 2015.

\bibitem{FeTiWe}
R.~Feng, G.~Tian, and D.~Wei.
\newblock Small gaps of {GOE}.
\newblock {\em Geom. Funct. Anal.}, 29(6):1794--1827, 2019.

\bibitem{FeWe}
R.~Feng and D.~Wei.
\newblock Small gaps of circular {$\beta$}-ensemble.
\newblock {\em Ann. Probab.}, 49(2):997--1032, 2021.

\bibitem{Fo2}
P.~J. Forrester.
\newblock The spectrum edge of random matrix ensembles.
\newblock {\em Nuclear Phys. B}, 402(3):709--728, 1993.

\bibitem{Fo}
P.~J. Forrester.
\newblock {\em Log-gases and random matrices}, volume~34 of {\em London
  Mathematical Society Monographs Series}.
\newblock Princeton University Press, Princeton, NJ, 2010.

\bibitem{F22}
P.~J. Forrester.
\newblock A review of exact results for fluctuation formulas in random matrix
  theory.
\newblock {\em Probab. Surv.}, 20:170--225, 2023.

\bibitem{GMW}
T.~Guhr, A.~M\"{u}ller-Groeling, and H.~A. Weidenm\"{u}ller.
\newblock Random-matrix theories in quantum physics: common concepts.
\newblock {\em Phys. Rep.}, 299(4-6):189--425, 1998.

\bibitem{Gu}
J.~Gustavsson.
\newblock Gaussian fluctuations of eigenvalues in the {GUE}.
\newblock {\em Ann. Inst. H. Poincar\'{e} Probab. Statist.}, 41(2):151--178,
  2005.

\bibitem{HoPa}
D.~Holcomb and E.~Paquette.
\newblock The maximum deviation of the {${\rm Sine}_\beta$} counting process.
\newblock {\em Electron. Commun. Probab.}, 23:Paper No. 58, 13, 2018.

\bibitem{HoVa}
D.~Holcomb and B.~Valk\'{o}.
\newblock Overcrowding asymptotics for the {$\rm sine_\beta$} process.
\newblock {\em Ann. Inst. Henri Poincar\'{e} Probab. Stat.}, 53(3):1181--1195,
  2017.

\bibitem{Ism}
M.~E.~H. Ismail.
\newblock An electrostatics model for zeros of general orthogonal polynomials.
\newblock {\em Pacific J. Math.}, 193(2):355--369, 2000.

\bibitem{israel15}
A.~Israel.
\newblock The eigenvalue distribution of time-frequency localization operators.
\newblock {\em arXiv preprint:1502.04404}.

\bibitem{Jo}
K.~Johansson.
\newblock On fluctuations of eigenvalues of random {H}ermitian matrices.
\newblock {\em Duke Math. J.}, 91(1):151--204, 1998.

\bibitem{KaRoDa}
S.~Karnik, J.~Romberg, and M.~A. Davenport.
\newblock Improved bounds for the eigenvalues of prolate spheroidal wave
  functions and discrete prolate spheroidal sequences.
\newblock {\em Appl. Comput. Harmon. Anal.}, 55:97--128, 2021.

\bibitem{KaZhWaRoDa}
S.~Karnik, Z.~Zhu, M.~B. Wakin, J.~Romberg, and M.~A. Davenport.
\newblock The fast {S}lepian transform.
\newblock {\em Appl. Comput. Harmon. Anal.}, 46(3):624--652, 2019.

\bibitem{KaAbMe}
A.~Karoui, I.~Mehrzi, and T.~Moumni.
\newblock Eigenfunctions of the {A}iry's integral transform: properties,
  numerical computations and asymptotic behaviors.
\newblock {\em J. Math. Anal. Appl.}, 389(2):989--1005, 2012.

\bibitem{KrML}
T.~Kriecherbauer and K.~T.-R. McLaughlin.
\newblock Strong asymptotics of polynomials orthogonal with respect to {F}reud
  weights.
\newblock {\em Internat. Math. Res. Notices}, (6):299--333, 1999.

\bibitem{Kuijlaars}
A.~Kuijlaars.
\newblock Lecture notes on {R}iemann-{H}ilbert problems and multiple orthogonal
  polynomials.
\newblock {\em Constructive Functions 2014, Nashville TN, USA,
  http://wis.kuleuven.be/analyse/arno}, 2014.

\bibitem{LaLeWe}
G.~Lambert, M.~Ledoux, and C.~Webb.
\newblock Quantitative normal approximation of linear statistics of
  {$\beta$}-ensembles.
\newblock {\em Ann. Probab.}, 47(5):2619--2685, 2019.

\bibitem{LP22}
G.~Lambert and E.~Paquette.
\newblock Strong approximation of {G}aussian {$\beta$} ensemble characteristic
  polynomials: the hyperbolic regime.
\newblock {\em Ann. Appl. Probab.}, 33(1):549--612, 2023.

\bibitem{MR222554}
H.~J. Landau.
\newblock Necessary density conditions for sampling and interpolation of
  certain entire functions.
\newblock {\em Acta Math.}, 117:37--52, 1967.

\bibitem{MR593228}
H.~J. Landau and H.~Widom.
\newblock Eigenvalue distribution of time and frequency limiting.
\newblock {\em J. Math. Anal. Appl.}, 77(2):469--481, 1980.

\bibitem{Le16}
T.~Lebl\'{e}.
\newblock Logarithmic, {C}oulomb and {R}iesz energy of point processes.
\newblock {\em J. Stat. Phys.}, 162(4):887--923, 2016.

\bibitem{LeRi}
M.~Ledoux and B.~Rider.
\newblock Small deviations for beta ensembles.
\newblock {\em Electron. J. Probab.}, 15:no. 41, 1319--1343, 2010.

\bibitem{LeLu}
A.~L. Levin and D.~S. Lubinsky.
\newblock {$L_\infty$} {M}arkov and {B}ernstein inequalities for {F}reud
  weights.
\newblock {\em SIAM J. Math. Anal.}, 21(4):1065--1082, 1990.

\bibitem{LeLu3}
A.~L. Levin and D.~S. Lubinsky.
\newblock Christoffel functions, orthogonal polynomials, and {N}evai's
  conjecture for {F}reud weights.
\newblock {\em Constr. Approx.}, 8(4):463--535, 1992.

\bibitem{LeLu4}
A.~L. Levin and D.~S. Lubinsky.
\newblock {$L_p$} {M}arkov-{B}ernstein inequalities for {F}reud weights.
\newblock {\em J. Approx. Theory}, 77(3):229--248, 1994.

\bibitem{LeLu2}
E.~Levin and D.~S. Lubinsky.
\newblock {\em Orthogonal polynomials for exponential weights}, volume~4 of
  {\em CMS Books in Mathematics/Ouvrages de Math\'{e}matiques de la SMC}.
\newblock Springer-Verlag, New York, 2001.

\bibitem{LeLu5}
E.~Levin and D.~S. Lubinsky.
\newblock Universality limits for exponential weights.
\newblock {\em Constr. Approx.}, 29(2):247--275, 2009.

\bibitem{LeLu6}
E.~Levin and D.~S. Lubinsky.
\newblock Universality limits at the soft edge of the spectrum via classical
  complex analysis.
\newblock {\em Int. Math. Res. Not. IMRN}, (13):3006--3070, 2011.

\bibitem{Lew22}
M.~Lewin.
\newblock Coulomb and {R}iesz gases: {T}he known and the unknown.
\newblock {\em J. Math. Phys.}, 63(6):Paper No. 061101, 77, 2022.

\bibitem{MR2006561}
N.~Marco, X.~Massaneda, and J.~Ortega-Cerd\`a.
\newblock Interpolating and sampling sequences for entire functions.
\newblock {\em Geom. Funct. Anal.}, 13(4):862--914, 2003.

\bibitem{MNT}
A.~M\'{a}t\'{e}, P.~Nevai, and V.~Totik.
\newblock Asymptotics for the zeros of orthogonal polynomials associated with
  infinite intervals.
\newblock {\em J. London Math. Soc. (2)}, 33(2):303--310, 1986.

\bibitem{MLMi}
K.~T.-R. McLaughlin and P.~D. Miller.
\newblock The {$\overline{\partial}$} steepest descent method for orthogonal
  polynomials on the real line with varying weights.
\newblock {\em Int. Math. Res. Not. IMRN}, pages Art. ID rnn 075, 66, 2008.

\bibitem{Meh}
M.~L. Mehta.
\newblock {\em Random matrices}, volume 142 of {\em Pure and Applied
  Mathematics (Amsterdam)}.
\newblock Elsevier/Academic Press, Amsterdam, third edition, 2004.

\bibitem{NaTr}
F.~Nakano and K.~D. Trinh.
\newblock Gaussian beta ensembles at high temperature: eigenvalue fluctuations
  and bulk statistics.
\newblock {\em J. Stat. Phys.}, 173(2):295--321, 2018.

\bibitem{MR2929058}
S.~Nitzan and A.~Olevskii.
\newblock Revisiting {L}andau's density theorems for {P}aley-{W}iener spaces.
\newblock {\em C. R. Math. Acad. Sci. Paris}, 350(9-10):509--512, 2012.

\bibitem{Ol}
F.~W.~J. Olver.
\newblock {\em Asymptotics and special functions}.
\newblock AKP Classics. A K Peters, Ltd., Wellesley, MA, 1997.
\newblock Reprint of the 1974 original [Academic Press, New York; MR0435697 (55
  \#8655)].

\bibitem{Os}
A.~Osipov.
\newblock Certain upper bounds on the eigenvalues associated with prolate
  spheroidal wave functions.
\newblock {\em Appl. Comput. Harmon. Anal.}, 35(2):309--340, 2013.

\bibitem{RaRiVi}
J.~A. Ram\'{\i}rez, B.~Rider, and B.~Vir\'{a}g.
\newblock Beta ensembles, stochastic {A}iry spectrum, and a diffusion.
\newblock {\em J. Amer. Math. Soc.}, 24(4):919--944, 2011.

\bibitem{SaTo}
E.~B. Saff and V.~Totik.
\newblock {\em Logarithmic potentials with external fields}, volume 316 of {\em
  Grundlehren der mathematischen Wissenschaften [Fundamental Principles of
  Mathematical Sciences]}.
\newblock Springer-Verlag, Berlin, 1997.
\newblock Appendix B by Thomas Bloom.

\bibitem{SV15}
K.~Schubert and M.~Venker.
\newblock Empirical spacings of unfolded eigenvalues.
\newblock {\em Electron. J. Probab.}, 20:Paper No. 120, 37, 2015.

\bibitem{Seip}
K.~Seip.
\newblock {\em Interpolation and sampling in spaces of analytic functions},
  volume~33 of {\em University Lecture Series}.
\newblock American Mathematical Society, Providence, RI, 2004.

\bibitem{Sh}
M.~Shcherbina.
\newblock Fluctuations of linear eigenvalue statistics of {$\beta$} matrix
  models in the multi-cut regime.
\newblock {\em J. Stat. Phys.}, 151(6):1004--1034, 2013.

\bibitem{Sh14}
M.~Shcherbina.
\newblock Change of variables as a method to study general {$\beta$}-models:
  bulk universality.
\newblock {\em J. Math. Phys.}, 55(4):043504, 23, 2014.

\bibitem{ShSe}
Z.~Shen and K.~Serkh.
\newblock On the evaluation of the eigendecomposition of the airy integral
  operator.
\newblock {\em Appl. Comput. Harmon. Anal.}, 57:105--150, 2022.

\bibitem{So}
A.~B. Soshnikov.
\newblock Gaussian fluctuation for the number of particles in {A}iry, {B}essel,
  sine, and other determinantal random point fields.
\newblock {\em J. Statist. Phys.}, 100(3-4):491--522, 2000.

\bibitem{SoWo}
P.~Sosoe and P.~Wong.
\newblock Local semicircle law in the bulk for {G}aussian {$\beta$}-ensemble.
\newblock {\em J. Stat. Phys.}, 148(2):204--232, 2012.

\bibitem{Sz}
G.~Szeg\H{o}.
\newblock {\em Orthogonal polynomials}.
\newblock American Mathematical Society Colloquium Publications, Vol. XXIII.
  American Mathematical Society, Providence, R.I., fourth edition, 1975.

\bibitem{Th}
E.~Thoma.
\newblock Overcrowding and separation estimates for the coulomb gas.
\newblock {\em Comm. Pure Appl. Math.}, 77(7):3227--3276, 2024.

\bibitem{TiRiKa}
P.~Tian, R.~Riser, and E.~Kanzieper.
\newblock Statistics of local level spacings in single- and many-body quantum
  chaos.
\newblock {\em Phys. Rev. Lett.}, 132:220401, May 2024.

\bibitem{Tr}
K.~D. Trinh.
\newblock Global spectrum fluctuations for {G}aussian beta ensembles: a
  {M}artingale approach.
\newblock {\em J. Theoret. Probab.}, 32(3):1420--1437, 2019.

\bibitem{VaVi}
B.~Valk\'{o} and B.~Vir\'{a}g.
\newblock Continuum limits of random matrices and the {B}rownian carousel.
\newblock {\em Invent. Math.}, 177(3):463--508, 2009.

\bibitem{Vi}
J.~P. Vinson.
\newblock {\em Closest spacing of consecutive eigenvalues}.
\newblock PhD thesis, Princeton University, 2001.

\bibitem{MR169015}
H.~Widom.
\newblock Asymptotic behavior of the eigenvalues of certain integral equations.
  {II}.
\newblock {\em Arch. Rational Mech. Anal.}, 17:215--229, 1964.

\bibitem{Wo}
P.~Wong.
\newblock Local semicircle law at the spectral edge for {G}aussian
  {$\beta$}-ensembles.
\newblock {\em Comm. Math. Phys.}, 312(1):251--263, 2012.

\bibitem{Zh}
C.~Zhong.
\newblock Large deviation bounds for the airy point process.
\newblock {\em arXiv preprint: 1910.00797}, 2020.

\end{thebibliography}

\end{document}